\documentclass[reqno,12pt]{amsart}

\usepackage[a4paper,  margin=3.2cm]{geometry}

\usepackage{amsmath,amssymb}
\usepackage{dsfont}         
\usepackage{comment}
\newtheorem{theorem}{Theorem}[section]

\newtheorem{lemma}[theorem]{Lemma}
\newtheorem{proposition}[theorem]{Proposition}
\theoremstyle{definition}
\newtheorem{definition}[theorem]{Definition}
\newtheorem{assumption}[theorem]{Assumption}
\theoremstyle{remark}
\newtheorem{remark}[theorem]{Remark}
\numberwithin{equation}{section}

\title[A Markov process for an age-structured population]{A Markov process for an infinite age-structured population}

\author{Dominika Jasi\'nska}

\address{Instytut Matematyki, Uniwersytet Marii Curie-Sk{\l}odowskiej, 20-031 Lublin, Poland}
\email{jasdominika@wp.pl}

\author{Yuri Kozitsky}

\address{Instytut Matematyki, Uniwersytet Marii Curie-Sk{\l}odowskiej, 20-031 Lublin, Poland}
\email{jkozi@hektor.umcs.lublin.pl}


\begin{document}

\begin{abstract}
For an infinite system of particles arriving in and departing from a
habitat $X$ -- a locally compact Polish space with a positive Radon
measure $\chi$ -- a Markov process is constructed in an explicit
way. Along with its location $x\in X$, each particle is
characterized by age $\alpha\geq 0$ -- time since arriving. As the
state space one takes the set of marked configurations
$\widehat{\Gamma}$, equipped with a metric that makes it a complete
and separable metric space. The stochastic evolution of the system
is described by a Kolmogorov operator $L$, expressed through the
measure $\chi$ and a departure rate $m(x,\alpha)\geq 0$, and acting
on bounded continuous functions $F:\widehat{\Gamma}\to \mathds{R}$.
For this operator, we pose the martingale problem and show that it
has a unique solution, explicitly constructed in the paper. We also
prove that the corresponding process has a unique stationary state
and is temporarily egrodic if the rate of departure is separated
away from zero.
\end{abstract}
\maketitle

\section{Introduction}

In recent  years, the stochastic dynamics of age-structured
populations attract considerable attention, mostly due to their
various applications, see \cite{He,Fima,Koz1,bib7,Wang}. Finite
populations of this kind are much more mathematically tractable,
 in contrast to infinite ones
where only few results are known.  In the present work, we construct
a Markov process for  an individual-based model of an infinite
population of point entities that arrive in and depart from a
continuous habitat $X\subseteq\mathds{R}^d$.

Let $X$ be a locally compact Polish space and $\chi$ a positive
Radon measure defined thereon. $X$ will serve as the habitat for
possibly infinite populations of entities (particles) that arrive in
and depart from $X$ at random. Along with the location $x\in X$,
each entity is characterized by age $\alpha\in
[0+\infty)=:\mathds{R}_{+}$ -- time of its presence in the
population.  The probability that an entity appears at time $t$ in a
given compact $\Lambda \subset X$ is set to be $1 - \exp( - t
\chi(\Lambda))$. Similarly, the probability that the entity located
at $x$ departs at time $t$ is $1- \exp(-t m(x,\alpha))$, where
$\alpha$ is its age at the moment of departure. Since we do not
assume that $\chi(X)$ be finite, the total population can get
instantly infinite, even if it is finite in the initial state. That
is, our model is characterized by an arriving measure $\chi$ and a
departing rate $m(x,\alpha)$.

To describe microscopic states of the population we use so called
\emph{marked configurations} $\hat{\gamma}$, cf. \cite[Definition
9.1.II, page 3]{DV2}. Set
\begin{equation*}
  \Gamma = \{ \gamma \subset X: |\gamma \cap \Lambda |<\infty, \
  {\rm for} \ {\rm each} \ {\rm compact} \ \Lambda \subset X\},
\end{equation*}
where $|\gamma \cap \Lambda |$ is the number of the elements of
$\gamma$ contained in $\Lambda$. Let $\mathcal{A}$ be the set of
finite configurations $a =\{\alpha_1 , \dots , \alpha_n\}\subset
[0,+\infty)$, $n\in \mathds{N}_0$. This set -- equipped with a
metric that makes it a complete and locally compact metric space --
serves as the space of marks for our model. That is, for $x\in X$,
by $a(x)\in \mathcal{A}$ we mean the ages of particles located at
$x$. Then a marked configuration $\hat{\gamma}$ is the pair
$(\gamma, a)$ with $\gamma\in \Gamma$ and $a:X\to \mathcal{A}$. In
this case, we also write $\gamma=p(\hat{\gamma})$ and call $\gamma$
ground configuration for $\hat{\gamma}$. The set of all marked
configurations $\widehat{\Gamma}$ is equipped with a metric, in
which it is complete and separable. Elements of $\hat{\gamma}$ are
pairs $(x, \alpha)=:\hat{x}$. They can be enumerated in
$\hat{\gamma}$, which allows one to consider sums $\sum_{\hat{x}\in
\hat{\gamma}} g(\hat{x})$ for continuous compactly supported
functions $g(\hat{x}) = g(x,\alpha)$.

The evolution of the population is governed by the (backward)
Kolmogorov equation
\begin{equation}
  \label{KE}
  \frac{d}{dt} F_t = L F_t, \qquad F_{t}|_{t=0} = F_0,
\end{equation}
where $t$ is time, $F_t$ is a suitable (test) function of
$\hat{\gamma}$ and the Kolmogorov operator $L$  is given by the
following expression
\begin{eqnarray}
  \label{S1}
 (LF)(\hat{\gamma}) & = & \sum_{(x,\alpha)\in \hat{ \gamma}} \frac{\partial}{\partial \alpha} F(\hat{\gamma})
  + \sum_{(x,\alpha) \in\hat{\gamma}} m(x, \alpha) \left[F(\hat{\gamma} \setminus (x,\alpha)) - F(\hat{\gamma}) \right]
  \\[.2cm] \nonumber
 & + & \int_{X}  \left[ F(\hat{\gamma} \cup (x,0)) - F(\hat{\gamma}) \right] \chi( d x).
 \nonumber
\end{eqnarray}
Here the first term describes aging, the second one corresponds to
departing, whereas the last term describes the appearance of new
entities.  In (\ref{S1}) and and in the sequel, in expressions like
$\hat{\gamma}\cup \hat{x}$ we consider $\hat{x}=(x,\alpha)$ as a
singleton configuration. The model parameters are subject to the
following
\begin{assumption}
  \label{Jass}
The departure rate $X\times \mathds{R}_{+}\ni (x,\alpha)\mapsto
m(x,\alpha)\in \mathds{R}_{+}$ is continuous and bounded, i.e., such
that $m(x,\alpha) \leq m_*$ for some $m_*>0$ and all $(x,\alpha)$.
Moreover, there exits $\varkappa: [0,1]\to \mathds{R}_{+}$  such
that $\varkappa(\epsilon) \to 0$ as $\epsilon \to 0$ and the
following holds
\begin{equation}
  \label{vkpa}
  \forall x\in X \qquad |m(x, \alpha) - m(x,\alpha')|\leq \varkappa
  (|\alpha-\alpha'|), \quad |\alpha-\alpha'|\in (0,1).
\end{equation}
The arriving measure $\chi$ is just a positive Radon measure.
\end{assumption}
The result of the present work can be outlined as follows. We
introduce a Banach space $\mathcal{C}$  of bounded continuous
functions $F:\widehat{\Gamma} \to \mathds{R}$, in which we define
$L$ as a closed and densely defined linear operator that satisfies
the conditions of the Hille-Yosida theorem, and hence is the
generator of a $C_0$-semigroup $\{S(t)\}_{t\geq 0}$. Then the
solution of (\ref{KE}) is obtained in the form $F_t = S(t) F_0$. For
a class of functions $\mathcal{F}_\Theta$, $F_t$ corresponding to
$F_0 \in \mathcal{F}_\Theta$ is obtained in an explicit form. This
allows us to explicitly construct the corresponding Markov
transition function $p^{\hat{\gamma}}_t$ and obtain
finite-dimensional laws of a Markov process $\mathcal{X}$ with
values in $\widehat{\Gamma}$, which describes the stochastic
evolution of our model. Possible objects of this kind are specified
as stochastic processes that solve the martingale problem for $L$.
Then we show that this problem is well-posed, i.e., uniqueness
holds. The main ingredient of the proof here is showing that the
corresponding Fokker-Planck equation for $L$ has a unique solution,
which we do by employing the resolvent of $L$. Assuming that
$m(x,\alpha ) \geq m_0
>0$, we also show that the process $\mathcal{X}$ has a unique
stationary state, explicitly constructed in the paper, such that the
laws of $\mathcal{X}(t)$ weakly converge to this state as $t\to
+\infty$.

\section{Preliminaries}
In this work, we use the following standard notions and notations.
For a Polish (=separable and completely metrizable) space $E$, by
$\mathcal{B}(E)$ we denote the corresponding Borel $\sigma$-field;
$C_{\rm b}(E)$ (resp. $B_{\rm b}(E)$) stands for the set of all
bounded and continuous (resp. bounded and measurable) functions
$f:E\to \mathds{R}$. By $C_{\rm b}^{+}(E)$ we denote the set of
positive elements of $C_{\rm b}(E)$. A subset $C_{\rm cs}(E) \subset
C_{\rm b}(E)$ consists of continuous compactly supported functions.
A family of functions $\mathcal{F}$ is said to \emph{separate} the
points of $E$ if for each distinct $x,y\in E$, one finds $f\in
\mathcal{F}$ such that $f(x)\neq f(y)$. By $\sigma\mathcal{F}$ we
denote the smallest sub-field of $\mathcal{B}(E)$ such that each
$f\in \mathcal{F}$ is $\sigma\mathcal{F}$-measurable; by
$\mathcal{P}(E)$ we mean the set of all probability Borel measures
on $(E, \mathcal{B}(E))$. For a given measure $\mu$ and a suitable
function $f$, we write $\mu(f) = \int f d\mu$. For a sequence
$\{\mu_n\}_{n\in \mathds{N}}\subset \mathcal{P}(E)$, by writing
$\mu_n \Rightarrow \mu \in \mathcal{P}(E)$ we mean its weak
convergence, i.e., $\mu_n (f) \to \mu(f)$ for all $f\in C_{\rm
b}(E)$. A family of functions $\mathcal{F}$ is said to be
\emph{separating} if $\mu_1 (f) = \mu_2(f)$ holding for all $f\in
\mathcal{F}$ implies $\mu_1=\mu_2$ for any $\mu_1,\mu_2\in
\mathcal{P}(E)$. If $\mathcal{F}$ separates the points of $E$ and
its linear span is an algebra with respect to pointwise operations,
then it is separating, see \cite[Theorem 4.5, page 113]{EK}. A
family of functions $\mathcal{F}$ is said to be \emph{convergence
determining} if $\mu_n (f) \to \mu(f)$ holding for all $f\in
\mathcal{F}$ implies $\mu_n \Rightarrow \mu$. For a suitable
$A\subset E$, by $\mathds{1}_A$ we denote the indicator of $A$,
i.e., the functions such that $\mathds{1}_A(x)=1$ if $x\in A$ and
$\mathds{1}_A(x)=0$ otherwise.

\subsection{The space of marks}

As mentioned above, our populations dwell in a locally compact
Polish space $X$. By $\Lambda$ we always denote a compact subset of
$X$. Each population member is characterized by its compound trait
$\hat{x} = (x,\alpha)\in \widehat{X}=X \times \mathds{R}_{+}$. For a
function, $g:\widehat{X}\to \mathds{R}$, we use interchangeable
writings $g(\hat{x})$ and $g(x,\alpha)$. The space $\widehat{X}$ is
equipped with the product topology assuming that the topology of
$\mathds{R}_{+}$ be defined by the metric which we introduce now.
For $\alpha \geq 0$, set $\omega(\alpha) = \min\{\alpha;
1/\alpha\}$, and then
\begin{equation}
  \label{Qh}
{r}(\alpha,0) = \omega(\alpha), \quad   {r}(\alpha, \alpha') =
\min\{|\alpha-\alpha'|; \omega(\alpha) + \omega(\alpha')\},
\end{equation}
where $|\beta|$ is the usual absolute value of $\beta\in
\mathds{R}$.
\begin{proposition}
  \label{Qh1pn}
The above introduced ${r}$  is a metric and $(\mathds{R}_{+}, {r})$
is a compact metric space.
\end{proposition}
\begin{proof}
To prove the first part we just have to check the validity of the
triangle inequality
\begin{equation}
  \label{QH}
{r}(\alpha_1, \alpha_2) \leq {r}(\alpha_1, \alpha_3)+ {r}(\alpha_2,
\alpha_3).
\end{equation}
This technical exercise is made in Appendix. To prove the
compactness, we have to show that: (a) ${r}$ is complete; (b) the
space $(\mathds{R}_{+}, {r})$ is totally bounded. Assume that
$\{\alpha_n\}_{n\in \mathds{N}}\subset \mathds{R}_{+}$ is an
${r}$-Cauchy sequence. Here one may have the following
possibilities: (i) there exists $\bar{\alpha}< \infty$ such that
$\alpha_n \leq\bar{ \alpha}$, $n\in \mathds{N}$; (ii) the considered
sequence contains a subsequence that diverges in the usual sense. In
case (i), $\{\alpha_n\}_{n\in \mathds{N}}$ contains a subsequence,
say $\{\alpha_{n_k}\}_{k\in \mathds{N}}$, such that $|\alpha_{n_k} -
\alpha_*|\to 0$ as $k\to +\infty$ for some $\alpha_* \leq
\bar{\alpha}$. At the same time, for $\varepsilon < 2/\bar{\alpha}$,
${r}(\alpha_n,\alpha_m) < \varepsilon$ implies $|\alpha_n-\alpha_m|
< \varepsilon$, see (\ref{Qh}), which means that ${r}(\alpha_n,
\alpha_*) \to 0$ as $n\to +\infty$. In case (ii), the divergent
subsequence converges in ${r}$ to zero, which implies that the whole
sequence converges to zero in ${r}$. Hence, the latter metric is
complete. To prove (b), we set $\tilde{B}_\varepsilon
(\alpha)=\{\alpha' \in \mathds{R}_{+}:
{r}(\alpha,\alpha')<\varepsilon\}$. Fix $\varepsilon\in(0,1)$ and
take the least $k\in \mathds{N}$ such that $k+1> 1/\varepsilon^2$.
Then $\mathds{R}_{+} =\cup_{j=0}^k \tilde{B}_\varepsilon
(j\varepsilon)$, which yields the property in question.
\end{proof}
Let us now compare $r$ with the absolute-value metrics of
$\mathds{R}_{+}$. By $C(\mathds{R}_{+},\mathcal{T}_r)$ we will mean
the sets of all bounded $r$-continuous functions, whereas $C_{\rm
b}(\mathds{R}_{+},\mathcal{T}_{|\cdot|})$ is going to stand for the
set of all bounded $|\cdot|$-continuous functions.
\begin{proposition}
  \label{Qh2pn}
$\mathcal{T}_{{r}}$ is coarser than $\mathcal{T}_{|\cdot|}$, and
hence the embedding $(\mathds{R}_{+},\mathcal{T}_{|\cdot|})
\hookrightarrow (\mathds{R}_{+},\mathcal{T}_{{r}})$ is continuous,
whereas both latter topological spaces are Borel isomorphic.
Moreover,
\begin{equation}
  \label{Th0}
C(\mathds{R}_{+},\mathcal{T}_{{r}})=\{u\in C_{\rm
b}(\mathds{R}_{+},\mathcal{T}_{|\cdot|}): \lim_{\alpha\to +\infty}
u(\alpha) = u(0)\}.
\end{equation}
\end{proposition}
\begin{proof}
The validity of the first statement and (\ref{Th0}) readily follows
by the fact that each $|\cdot|$-convergent sequence is also
${r}$-convergent, and each ${r}$-convergent sequence either
 converges in $|\cdot|$ to the same limit $\alpha \neq 0$, or has two
$|\cdot|$-accumulation points: $0$ and $\infty$. Since the mentioned
embedding is continuous and injective, it is also Borel-measurable.
By Kuratowski's theorem \cite{Part}, its inverse is also measurable
and thus is the isomorphism in question. This, in particular, means
that the corresponding Borel $\sigma$-fields coincide.
\end{proof}
\begin{definition}
  \label{Ju1df}
For a suitable $u \in C (\mathds{R}_{+}, \mathcal{T}_{r})$, we
introduce the map $\mathds{R}_{+} \ni \alpha \mapsto u'(\alpha)$
meaning the usual derivative if $\alpha >0$, and the right-hand side
one if $\alpha=0$. A given $u$ is said to be continuously
differentiable (on $\mathds{R}_{+}$) if $u'\in C (\mathds{R}_{+},
\mathcal{T}_{r})$.
\end{definition}
Let us consider the following functions
\begin{equation}
  \label{uM1}
u_n(\alpha) = \frac{\alpha^2}{1+n \alpha^3}, \qquad \alpha\in
\mathds{R}_{+}, \quad n\in \mathds{N}.
\end{equation}
It is clear that: (a) each $u_n$ is continuously differentiable, see
Definition \ref{Ju1df}; (b) $u_n$ is decreasing for
$\alpha>(2/n)^{1/3}$. Moreover, $u_n(\alpha) \leq 2^{2/3}/3 n^{2/3}$
and
\begin{equation}
  \label{uM1a}
u_n'(\alpha) = \frac{2\alpha - n\alpha^4}{(1+n \alpha^3)^2}, \qquad
\left|u_n'(\alpha) \right| \leq c/n^{1/3},
\end{equation}
the latter holding for some $c>0$ and all $\alpha\geq 0$.

Now let $\{\sigma_k\}_{k\in \mathds{N}}=:\varSigma\subset
[0,+\infty)$ be countable and such that: (i) $\sigma_1 =0$; (ii)
$\sigma_k< \sigma_{k+1}$ for all $k\in \mathds{N}$; $\sigma_k \to
\bar{\sigma}<\infty$ as $n\to +\infty$. Next, for $k,n\in
\mathds{N}$, we set
\begin{equation}
  \label{uM2}
  w_{k,n}(\alpha) = e^{-\sigma_k u_n(\alpha)}.
\end{equation}
Then $w_{k,n}$ is continuously differentiable and the following
holds
\begin{equation}
  \label{uM2z}
  \left|w'_{k,n}(\alpha)\right| \leq \frac{\bar{\sigma} c}{n^{1/3}}
  w_{k,n}(\alpha),
\end{equation}
where $c$ is the same as in (\ref{uM1a}).

Now let $a$ be a finite collection of points $\alpha_l\in
\mathds{R}_{+}$. That is, $a=\{\alpha_l\}_{1\leq l\leq m}$,
$\alpha_l\leq \alpha_{l+1}$ for all $l$. For $\alpha\in a$, by
$n_a(\alpha)\in \mathds{N}$ we will denote the multiplicity of
$\alpha$ in $a$, i.e., the number of elements of $a$ coinciding with
this $\alpha$. We extend it to all $\alpha\geq 0$ by setting
$n_a(\alpha)=0$ whenever $\alpha$ is not in $a$. Two such $a$ and
$a'$ are equal if they consist of exactly the same elements, with
the same multiplicities.
\begin{proposition}
  \label{uM1pn}
 Let $a$ and $a'$ be as just described. Then they are equal
if
\begin{equation}
   \label{uM3}
 \sum_{\alpha \in a} w_{k,n}(\alpha) = \sum_{\alpha \in a'}
 w_{k,n}(\alpha),
 \end{equation}
 holding for all $k,n\in \mathds{N}$.
 \end{proposition}
\begin{proof}
For $a$ as above and $\zeta \in \mathds{C}$, consider
\begin{equation*}
  f_{n,a}(\zeta) = \sum_{\alpha\in a} \exp\left( - \zeta
  u_n(\alpha)\right), \qquad n\in \mathds{N}.
\end{equation*}
Each such $f$ is an exponential type entire function. By (\ref{uM2})
and (\ref{uM3}) we have that $(f_{n,a}- f_{n,a'})|_{\varSigma}=0$,
holding for all $n\in \mathds{N}$. Since $\varSigma$ has a limiting
point, this implies $f_{n,a}(\zeta)= f_{n,a'}(\zeta)$ for all $\zeta
\in \mathds{R}$ and $n\in  \mathds{N}$. Obviously, $\lim_{\zeta \to
+\infty} f_{n,a}(\zeta)=n_a(0)$, where $n_a(0)\geq 0$ is the
multiplicity of $\alpha =0$ in $a$. Then the just mentioned equality
yields $n_a(0)=n_{a'}(0)$, and also
\begin{equation}
  \label{uM5}
   \sum_{\alpha\in a\setminus \{0\}} \exp\left( - \zeta
  u_n(\alpha)\right)= \sum_{\alpha\in a'\setminus \{0\}} \exp\left( - \zeta
  u_n(\alpha)\right).
\end{equation}
Let $\alpha_*$ and $\alpha'_*$ be the least positive elements of $a$
and $a'$, respectively. Take $n> 2/\alpha^3_{\flat}$,
$\alpha_\flat:= \min\{\alpha_*; \alpha'_*\}$. Then, for such $n$ and
all $\alpha>\alpha_\flat$, one has $u_n(\alpha_\flat) >
u_n(\alpha)$. Now we multiply both sides of (\ref{uM5}) by $e^{\zeta
u_n(\alpha_\flat)}$ and pass to the limit $\zeta \to -\infty$. This
yields that $\alpha_* = \alpha'_*$ and $n_a(\alpha_*) =
n_{a'}(\alpha'_*)$. Thereafter, we subtract the coinciding terms
from both sides of (\ref{uM5}) and proceed to comparing the
remaining least elements of $a$ and $a'$. This eventually yields the
equality to be proved.
\end{proof}
Let $\mathcal{A}$ be the set of all $a=\{\alpha_l\}_{1\leq l \leq
m}$, $m\in \mathds{N}_0$, $0\leq \alpha_1 \leq \alpha_2\leq \cdots
\leq \alpha_m$. Define
\begin{gather}
  \label{uM6}
  \rho(a, a') = \sum_{k, n\in
  \mathds{N}}\frac{2^{-k-n} \rho_{k,n}(a, a')}{1+ \rho_{k,n}(a,
  a')}, \\[.2cm] \nonumber \rho_{k,n}(a, a'):=
  \left|\sum_{\alpha \in a} w_{k,n}(\alpha) - \sum_{\alpha \in a'} w_{k,n}(\alpha)
  \right|.
\end{gather}
By Proposition \ref{uM1pn} it follows that $\rho$ is a metric on
$\mathcal{A}$. Each $a\in \mathcal{A}$ can be considered as a finite
counting measure defined on the compact space $(\mathds{R}_{+}, r)$,
for which $a(\Delta) = \sum_{\alpha\in a} \mathds{1}_\Delta (a)=
|a\cap\Delta|$, holding for all Borel subsets $\Delta$. The weak
topology of $\mathcal{A}$ is defined as the coarsest topology that
makes continuous all the maps $a\mapsto \sum_{\alpha\in a}w(a)$,
$w\in C(\mathds{R}_{+},\mathcal{T}_r)$. In the weak topology,
$\mathcal{A}$ is a closed subset of the space of all finite positive
measures on $(\mathds{R}_{+},\mathcal{T}_r)$.
\begin{proposition}
  \label{uM2pn}
$(\mathcal{A},\rho)$ is a complete metric space. The corresponding
metric topology coincides with the weak topology that turns
$\mathcal{A}$ into a locally compact Polish space.
\end{proposition}
\begin{proof}
As each $w_{k,n}$ is in $C(\mathds{R}_{+},\mathcal{T}_r)$, the weak
convergence of a sequence $\{a_m\}_{m\in \mathds{N}}\subset
\mathcal{A}$ to a certain $a\in \mathcal{A}$ yields $\rho(a,a_m) \to
0$, $m\to +\infty$. Assume now that $\{a_m\}_{m\in \mathds{N}}$ is a
$\rho$-Cauchy sequence. By taking $\sigma=0$ we then get from the
latter that, for some $m_*\in \mathds{N}$, the cardinalities of all
$a_m$, $m>m_*$, coincide. By Prohorov's theorem this yields that
$\{a_m\}_{m\in \mathds{N}}$ contains a subsequence that weakly
converges to some $a$. Hence, the whole sequence converges in $\rho$
to this $a$. Then the metric is complete and the corresponding
metric topology is exactly the weak topology of $\mathcal{A}$. The
separability and local compactness follow by the fact that
$(\mathds{R}_{+},\mathcal{T}_r)$ is compact.
\end{proof}

\subsection{Marked configuration spaces}

Let ${\Gamma}$ be the set of all locally finite simple
configurations on $X$. That is, each ${\gamma}\in {\Gamma}$ is a
subset of $X$ such that each compact $\Lambda \subset X$ contains a
finite number of the elements of ${\gamma}$. Let now
$\breve{\gamma}$ be the pair $({\gamma},n)$, ${\gamma}\in {\Gamma}$
and $n:{\gamma}\to \mathds{N}$. The value of $n$ at a given $x\in
{\gamma}$ can be considered as the multiplicity of $x$ in
$\breve{\gamma}$. That is, $\breve{\gamma}$ is a configuration with
multiple locations, for which ${\gamma}$ is the ground
configuration. Sometimes, we will write $n_{\breve{\gamma}}(x)$ to
explicitly indicate that we mean the multiplicity of $x$ in the
mentioned $\breve{\gamma}$. By $\breve{\Gamma}$ we denote the set of
all such multiple configurations. For $\breve{\gamma}=(\gamma,n)$,
we write ${\gamma}= p(\breve{\gamma})$. The weak-hash (vague)
topology of $\breve{\Gamma}$ is defined as the coarsest topology
that makes continuous all the maps $\breve{\gamma}\mapsto \sum_{x\in
p(\breve{\gamma})} n(x) g(x)$, $g\in C_{\rm cs}(X)$. It is
well-known, see e.g., \cite[Lemma 1.2]{Zessin}, that with this
topology $\breve{\Gamma}$ is a Polish space, whereas $\Gamma$ is a
$G_\delta$ subset of $\breve{\Gamma}$, by which it is also Polish.
Following \cite{Lenard} we will also consider $\breve{\gamma}$ as
configurations of point particles, in which distinct particles may
have the same location. Such particles can be enumerated, which
allows one to write
\begin{equation}
  \label{uM7}
  \sum_{x\in p(\breve{\gamma})} n(x) g(x) = \sum_{x\in
  \breve{\gamma}} g(x),
\end{equation}
where in the second sum we mean a certain enumeration of this sort.
In the same sense, we will write
\begin{equation*}
  \sum_{{x}_1\in \breve{\gamma}} \sum_{{x}_2\in \breve{\gamma}\setminus
  x_1} \cdots \sum_{x_m\in \breve{\gamma}\setminus \{x_1, \dots ,
  x_{m-1}\}} g(x_1, \dots , x_m), \qquad m\in \mathds{N},
\end{equation*}
where in expressions like $\breve{\gamma}\setminus x$ we treat $x$
the singleton $\{x\}$.

It is known, see, e.g., \cite[page 397]{Zessin}, that there exists a
collection $\{v_s\}_{s\in\mathds{N}} =:\mathcal{V}\subset C^{+}_{\rm
cs}(X)$ of suitable functions such that the metric
\begin{equation}
  \label{uM9}
  d(\breve{\gamma}, \breve{\gamma}') = \sum_{s\in \mathds{N}}
 \frac{2^{-s}d_s(\breve{\gamma}, \breve{\gamma}')}{1+d_s(\breve{\gamma},
 \breve{\gamma}')}, \quad d_s(\breve{\gamma},
 \breve{\gamma}'):= \left|\sum_{x\in \breve{\gamma}} v_s (x) - \sum_{x\in \breve{\gamma}'} v_s (x)
 \right|,
\end{equation}
is complete and consistent with the weak-hash topology of
$\breve{\Gamma}$. In the sequel, we will always mean this topology
of $\breve{\Gamma}$. Obviously, we can and will assume that
$\mathcal{V}$ contains also the following functions. Let $\delta$ be
a complete metric of $X$ and $X'$ a countable dense subset of $X$.
Each $x'\in X'$ has a countable base of compact neighborhoods, which
we denote by $D(x')$. Each $\Delta \in D(x')$ contains balls $B_q
(x')=\{x\in X: \delta (x,x') < q\}$ with compact closures, where $q$
is a rational number satisfying $q\leq q'$ for a $\Delta$-specific
$q'\in \mathds{Q}$. For $x'\in X'$, $\Delta \in D (x')$, $q\leq q'$
and $\varsigma \in (0,1)\cap\mathds{Q}$, let $v\in C^{+}_{\rm
cs}(X)$ be such that: (a) $v(x) \equiv \varsigma$ for $x\in B_q
(x')$; (b) $v(x) =0$ for $x\in X \setminus \Delta$. The countable
set of all such functions is supposed to be a part of $\mathcal{V}$,
and hence they are taken into account in (\ref{uM9}). Since each
$v_s$ has compact support, for each compact $\Lambda\subset X$ and
any two configurations, $d_s(\breve{\gamma}\cap\Lambda,
 \breve{\gamma}'\cap\Lambda)> 0$ only for finitely many $s$. Here
$\breve{\gamma}\cap \Lambda := (p(\breve{\gamma})\cap \Lambda, n)$.

For $\gamma\in \Gamma$, let $a:\gamma\to \mathcal{A}$ be a map, for
which we denote
\begin{equation}
  \label{uM10}
  |a(x)| = \sum_{\alpha\in a(x)} n_{a(x)} (\alpha).
\end{equation}
Then the pair $\hat{\gamma}=(\gamma, a)$ is a marked configuration
whose ground configuration is $\gamma$ and the mark map is $a$. By
writing $\hat{x} = (x,\alpha)\in \hat{\gamma}$ we will mean that
$x\in\gamma$ and $\alpha \in a(x)$.
 The configuration of marks $a(x) = \{ \alpha_1 , \dots ,
\alpha_{|a(x)|}\}$ yields the ages of the particles located at $x\in
\gamma$, whereas $|a(x)|$ is the total number of such particles. In
some cases, we write $a_{\hat{\gamma}}$ to indicate that $a$ is
defined on a given $\hat{\gamma}$. Let $\widehat{\Gamma}$ denote the
set of all marked configurations $\hat{\gamma}$. Let also
$\breve{p}:\widehat{\Gamma}\to \breve{\Gamma}$ be the map such that
$\breve{p}(\gamma, a)= (\gamma, |a|)$, where $|a|(x) = |a(x)|$ see
(\ref{uM10}). Then $p\circ \breve{p}$ maps $\hat{\gamma}=(\gamma,a)$
into its ground configuration $\gamma$. For brevity, by writing
$p(\hat{\gamma})$ we will mean $(p\circ \breve{p})(\hat{\gamma})$.
Our aim now is to equip $\widehat{\Gamma}$ with a complete metric.
 Define
\begin{gather}
  \label{uM12}
  \kappa(\hat{\gamma}, \hat{\gamma}') = \sum_{s,k,n\in \mathds{N}}
  \frac{2^{-(s+k+n)} \kappa_{s,k,n}(\hat{\gamma}, \hat{\gamma}')}{1+\kappa_{s,k,n}(\hat{\gamma},
  \hat{\gamma}')}, \\[.2cm] \nonumber \kappa_{s,k,n}(\hat{\gamma}, \hat{\gamma}'):=
  \left|\sum_{x\in p(\hat{\gamma})}v_s(x) \sum_{\alpha \in a_{\hat{\gamma}}(x)} w_{k,n}(\alpha) -
  \sum_{x\in  p(\hat{\gamma}')} v_s(x) \sum_{\alpha \in a_{\hat{\gamma}'}(x)} w_{k,n}(\alpha)
  \right|.
\end{gather}
Note that the latter can also be written as, cf. (\ref{uM7}),
\begin{gather}
  \label{uM12z}
\kappa_{s,k,n}(\hat{\gamma}, \hat{\gamma}') = \left|\sum_{\hat{x}\in
\hat{\gamma}} g_{s,k,n} (\hat{x}) - \sum_{\hat{x}\in \hat{\gamma}'}
g_{s,k,n} (\hat{x})  \right|, \\[.2cm] \nonumber g_{s,k,n} (x,\alpha):= v_s(x)
w_{k,n}(\alpha).
\end{gather}
For a compact $\Lambda \subset X$, we write $\hat{\gamma}\cap\Lambda
= (p(\hat{\gamma})\cap \Lambda, a)$, where $a$ is the restriction of
$a$ from $\hat{\gamma}$ to $p(\hat{\gamma})\cap \Lambda$.
\begin{proposition}
  \label{uM3pn}
For each $\varepsilon>0$, one may find a compact
$\Lambda_\varepsilon \subset X$ such that, for any two
configurations, the following holds
\begin{equation}
  \label{uM12a}
  \left| \kappa(\hat{\gamma},
\hat{\gamma}') - \kappa(\hat{\gamma}\cap\Lambda_\varepsilon,
\hat{\gamma}'\cap\Lambda_\varepsilon)\right| < \varepsilon.
\end{equation}
\end{proposition}
\begin{proof}
Fix $\varepsilon>0$ and then pick $s_*\in \mathds{N}$ such that
$2^{s_*}> 1/\varepsilon$. Now let $\Lambda_\varepsilon$ be covered
by the supports of $v_s$ with $s\leq s_*$. For such $s$ and all
$k,n\in \mathds{N}$, we have $\kappa_{s,k,n}(\hat{\gamma},
\hat{\gamma}')=\kappa_{s,k,n}(\hat{\gamma}\cap\Lambda_\varepsilon,
\hat{\gamma}'\cap\Lambda_\varepsilon)$, see (\ref{uM12}).  This
clearly yields (\ref{uM12a}).
\end{proof}
Since $\sigma=0$ is in $\varSigma$,  by (\ref{uM2}) and (\ref{uM9})
we have that
\begin{equation}
  \label{uM13}
  d (\breve{p}(\hat{\gamma}), \breve{p}(\hat{\gamma}')) \leq \kappa(\hat{\gamma},
  \hat{\gamma}').
\end{equation}
\begin{proposition}
\label{Qq1pn} The metric space $(\widehat{\Gamma}, \kappa)$ is
complete.
\end{proposition}
\begin{proof}
We begin by pointing out the following evident fact
\begin{equation}
  \label{uJ}
\tilde{\kappa}_s (\hat{\gamma}, \hat{\gamma}') := \sum_{k,n\in
\mathds{N}}
  \frac{2^{-(s+k+n)} \kappa_{s,k,n}(\hat{\gamma}, \hat{\gamma}')}{1+\kappa_{s,k,n}(\hat{\gamma},
  \hat{\gamma}')}\leq {\kappa} (\hat{\gamma}, \hat{\gamma}'),
\end{equation}
holding for all $s\in \mathds{N}$ and $\hat{\gamma}, \hat{\gamma}'$.
Let now $\{\hat{\gamma}_m=(\gamma_m, a_m)\}_{m\in \mathds{N}}\subset
\widehat{\Gamma}$ be a $\kappa$-Cauchy sequence. By (\ref{uM13}) the
sequence $\{\breve{p}(\hat{\gamma}_m)\}_{m\in \mathds{N}}\subset
\breve{\Gamma}$ converges to some $\breve{\gamma}$. Take now $x\in
p(\breve{\gamma})$ and then pick a compact $\Delta\subset X$ such
that $\Delta \cap p(\breve{\gamma})= \{x\}$. For this $\Delta$, we
then set $n_m(x)=\sum_{y\in
p(\breve{\gamma}_m)\cap\Delta}n_{\breve{\gamma}_m}(y)$, $m\in
\mathds{N}$. From the convergence of
$\{\breve{p}(\hat{\gamma}_m)\}_{m\in \mathds{N}}$ to
$\breve{\gamma}$, it follows that $n_m(x)\to n(x)$; hence, there
exists $m_*\in \mathds{N}$ such that $n_m(x)= n(x)$ for all $m>m_*$.
Now we pick $x'\in X'$ and $q \in Q$ such that $x\in B_{q/2}(x')$
and the closure of $B_q(x')$ lies in $\Delta$. Let now $v_s\in
\mathcal{V}$ be such that $v_s(y) =\varsigma \in (0,1)\cap
\mathds{Q}$, $y\in B_{q/2}(x')$, and $v_s(y) =0$ for $y\in
X\setminus B_{q}(x')$. For these $m_*$ and $s$,
$\{\hat{\gamma}_m\}_{m\geq m_*+1}$ is also a
$\tilde{\kappa}_s$-Cauchy sequence, see (\ref{uJ}), for which we
have
\begin{equation*}
  \kappa_{s,k,n}(\hat{\gamma}_{m},\hat{\gamma}_{m+l}) = \varsigma
  \left|\sum_{y\in p(\hat{\gamma}_{m}) \cap\Delta}\sum_{\alpha \in \hat{a}_m (y)} w_{\sigma_{k},n}(\alpha) -
   \sum_{y\in p(\hat{\gamma}_{m+l}) \cap\Delta}\sum_{\alpha \in \hat{a}_{m+l}
   (y)}
   w_{\sigma_{k},n}(\alpha)\right|,
\end{equation*}
holding for all $k,n\in \mathds{N}$, $m>m_*$ and $l\in \mathds{N}$.
Here $\hat{a}_{m}:=a_{\hat{\gamma}_m}$,
$\hat{a}_{m+l}:=a_{\hat{\gamma}_{m+l}}$. Let us enumerate
$\hat{x}=(x, \alpha)\in \hat{\gamma}_m\cap\Delta$ in such a way that
$\alpha_{p,m}\leq \alpha_{p+1,m}$ for all $p$. This yields
$\hat{\gamma}_m\cap\Delta=\{(x_{1,m} , \alpha_{1,m}), \dots
(x_{n,m}, \alpha_{n,m})\}$ with $n = n_{m}(x)=n(x)$. Similarly, we
have $\hat{\gamma}_{m+l}\cap\Delta=\{(x_{1,m+l} , \alpha_{1,m+l}),
\dots (x_{n,m+l}, \alpha_{n,m+l})\}$ with the same $n$. Then
$\{\alpha_{1,m}, \dots \alpha_{n,m}\} =:a_m \in \mathcal{A}$, and
also $\{\alpha_{1,m+l}, \dots \alpha_{n,m+l}\} =:a_{m+l}\in
\mathcal{A}$, and the latter equality can be rewritten as follows
\begin{equation*}
\kappa_{s,k,n}(\hat{\gamma}_{m},\hat{\gamma}_{m+l}) = \varsigma
\left| \sum_{p=1}^n w_{\sigma_{k},n}(\alpha_{p,m}) -\sum_{p=1}^n
w_{\sigma_{k},n}(\alpha_{p,m+l})\right| = \varsigma \rho_{k,n} (a_m,
a_{m+l}),
\end{equation*}
see (\ref{uM6}). By (\ref{uJ}) and (\ref{uM6}) we then get
\[
\rho(a_m, a_{m+l}) \leq (2^s/\varsigma) \tilde{\kappa}_s
(\hat{\gamma}_m, \hat{\gamma}_{m+l}).
\]
By Proposition \ref{uM2pn} this yields the convergence of
$\{a_m\}_{m>m_*}$ to some $a(x)\in \mathcal{A}$, which holds for
each $x\in \breve{\gamma}$. This defines the map
$a:p(\breve{\gamma})\to \mathcal{A}$, and hence the configuration
$\hat{\gamma}=(p(\breve{\gamma}),a)$. Our aim now is to prove that
$\kappa (\hat{\gamma}_m, \hat{\gamma})\to 0$ as $m\to +\infty$.

Fix $\varepsilon >0$ and then pick a compact $\Lambda_\varepsilon
\subset X$ such that (\ref{uM12a}) holds with $\varepsilon /3$ in
the right-hand side. Let $\Lambda_\varepsilon^o$ be its interior.
Then pick compact mutually disjoint $\Delta_x\subset \Lambda^o$,
$x\in p(\breve{\gamma})\cap \Lambda^o$ such that
$p(\breve{\gamma})\cap \Delta_x =\{x\}$.  As $\Lambda_\varepsilon$
is compact, $p(\breve{\gamma})\cap \Lambda_\varepsilon^o$ is finite.
Let $\{x_j\}_{j\leq J}$ be an enumeration of it. For brevity, we
will write $\Delta_j$ in place of $\Delta_{x_j}$, $j=1, \dots , J$.
Similarly as above, by the convergence of
$\{{\breve{p}(\hat{\gamma}_m)}\}_{m\in \mathds{N}}$ to
$\breve{\gamma}$, one finds $m_*$ such that $p(\hat{\gamma}_m)\cap
\Delta_j$ is a singleton and $|\breve{p}(\hat{\gamma}_m)\cap
\Delta_j| =: n_m(x_j)  = n(x_j)$, holding for all $m>m_*$ and $j\leq
J$.  Now we repeat the construction just made in each of
$\Delta_{j}$. That is, we enumerate $\hat{\gamma}_m \cap\Delta_{j}
=\{(x^j_{1,m}, \alpha^j_{1,m}), \dots ,
(x^j_{n(x_j),m},\alpha^j_{n(x_j),m})\}$, and then set $a_m^j
=\{\alpha^j_{1,m}, \dots , \alpha^j_{n(x_j),m}\}$. Then we set
$\hat{\gamma}_{*,m}=(\gamma_{*,m}, a_{*,m})$, where $\gamma_{*,m} =
p(\breve{\gamma})\cap\Lambda_\varepsilon^o=\{x_1, \dots  , x_J\}$
and $a_{*,m}(x_j) = a^j_m$. In other words, the ground configuration
of $\hat{\gamma}_{*,m}$ is the part of the limiting configurations
$p(\hat{\gamma})$ contained in $\Lambda_\varepsilon^o$, whereas the
marks are taken from the corresponding part of $\hat{\gamma}_m$. By
the triangle inequality we then have
\begin{equation}
  \label{Qq3}
 \kappa(\hat{\gamma}, \hat{\gamma}_{m}) \leq \kappa(\hat{\gamma}\cap\Lambda^o_\varepsilon, \hat{\gamma}_{*,m})
+\kappa(\hat{\gamma}_{m}\cap\Lambda_\varepsilon^o,
\hat{\gamma}_{*,m})+\varepsilon/3.
\end{equation}
By (\ref{uM12}), for each $s,k,n\in \mathds{N}$, we have
\begin{eqnarray*}
\kappa_{s,k,n} (\hat{\gamma}\cap\Lambda^o_\varepsilon,
\hat{\gamma}_{*,m})& = & \left|\sum_{j=1}^J v_s(x_j)
\left[\sum_{\alpha \in a^j_m} w_{\sigma_k,n}(\alpha) - \sum_{\alpha
\in a(x_j)}
w_{\sigma_k,n}(\alpha)\right] \right|\\[.2cm]& \leq & J \max_{j\leq J}
\rho_{k,n} (a_m^j, a(x_j)). \nonumber
\end{eqnarray*}
Likewise,
\begin{eqnarray*}
\kappa_{s,k,n} (\hat{\gamma}_m\cap\Lambda^o_\varepsilon,
\hat{\gamma}_{*,m})& = & \bigg{|}\sum_{j=1}^J
\left(\sum_{p=1}^{n(x_j)} v_s(x^j_{p,m}) \right) \sum_{\alpha \in
a^j_m} w_{\sigma_k,n}(\alpha) \\[.2cm] \nonumber& - & \sum_{j=1}^J v_s(x_j)
\sum_{\alpha \in a^j_m)}
w_{\sigma_k,n}(\alpha) \bigg{|}\\[.2cm]& \leq &  d_s (\breve{\gamma}_m\cap\Lambda^o_\varepsilon,
\breve{\gamma}\cap\Lambda^o_\varepsilon). \nonumber
\end{eqnarray*}
Both latter estimates yield
\[
\kappa(\hat{\gamma}\cap\Lambda^o_\varepsilon, \hat{\gamma}_{*,m})
\leq J\max_{j\leq J} \rho(a^j_m, a(x_j)),
\]
\[
\kappa(\hat{\gamma}_{m}\cap\Lambda_\varepsilon^o,
\hat{\gamma}_{*,m}) \leq d(\breve{\gamma}_m, \breve{\gamma}).
\]
By the aforementioned convergence $\breve{\gamma}_m \to
\breve{\gamma}$ and $a^j_m \to a(x_j)$, one can find $m_\varepsilon
> m_*$ such that the first two summands in (\ref{Qq3}) are smaller
than $\varepsilon$ for $m>m_\varepsilon$, which completes the proof.
\end{proof}

\subsection{Measures and functions on configuration spaces}

For $v_s$ and $w_{k,n}$ as in (\ref{uM12}) we set
\begin{equation*}
  \theta_{s,k,n} (x,\alpha ) = \exp\left( - v_s(x) w_{k,n} (\alpha)\right) -1
  = \exp\left( - g_{s,k,n}(x,\alpha)\right) -1,
\end{equation*}
see (\ref{uM12z}). Then $\theta_{s,k,n} (x,\alpha ) \in C_{\rm
cs}(\widehat{X})$ and $\theta_{s,k,n} (x,\alpha )\in (-1, 0]$. Let
$\Theta$ be the subset of $C_{\rm cs}(\widehat{X})$ consisting of
\begin{equation}
  \label{JJ}
  \theta (\hat{x}) = e^{- g(\hat{x})} - 1, \qquad g(x,\alpha) =
  \sum_{j} v_{s_j} (x) w_{k_j,n_j}(\alpha),
\end{equation}
where the latter sum runs over a finite subset of $\mathds{N}^3$.
That is, each such $g$ is a finite sum of $g_{s,k,n}$ defined in
(\ref{uM12z}). Note that $\Theta$ is countable and closed under the
map $(\theta,\theta') \mapsto (\theta \ast \theta')$, where
\begin{equation}
  \label{J1}
  (\theta \ast \theta')(\hat{x}) = \theta (\hat{x}) +
  \theta'(\hat{x}) + \theta (\hat{x})
  \theta'(\hat{x}) = -1 + (1+\theta(x)) (1+\theta'(x)).
\end{equation}
Moreover, by (\ref{uM1}), (\ref{uM2}) and (\ref{JJ}) it follows that
\begin{equation}
  \label{N11jj}
  g(x,\alpha) \leq g(x,0),
\end{equation}
holding for all $\alpha \geq 0$ and $x\in X$.
 Now for $\theta\in \Theta$, we set
\begin{equation}
  \label{N10j}
  F^\theta(\hat{\gamma}) = \prod_{\hat{x}\in \hat{\gamma}} (1 +
  \theta(\hat{x})) = \exp\left(- \sum_{\hat{x}\in \hat{\gamma}} g(\hat{x}) \right), \qquad \hat{\gamma}\in \widehat{\Gamma}.
\end{equation}
Then $F^\theta(\hat{\gamma})\in (0, 1]$ for all $\hat{\gamma}\in
\widehat{\Gamma}$, and hence $F^\theta\in C_{\rm
b}(\widehat{\Gamma})$. The set of all such functions will be denoted
by $\mathcal{F}_\Theta$. For $\mu \in
\mathcal{P}(\widehat{\Gamma})$, we then have
\[
\mu(F^\theta) = \int_{\widehat{\Gamma}} F^\theta (\hat{\gamma}) \mu
( d \hat{\gamma}) \leq 1.
\]
The Poisson measure $\pi_\varrho$ with intensity measure $\varrho$
satisfies
\begin{equation}
  \label{N7j}
  \pi_\varrho(F^\theta) = \exp\left( \varrho(\theta)\right) =
  \exp\left( \int_{\widehat{X}} \theta (\hat{x})
  \varrho(d\hat{x})\right).
\end{equation}
For $\mu_1,\mu_2\in\mathcal{P}(\widehat{\Gamma})$, their convolution
is defined by the expression
\begin{equation}
 \label{Conv}
  (\mu_1 \star \mu_2) (F) = \int_{\widehat{\Gamma}^2}
  F(\hat{\gamma}_1\cup \hat{\gamma}_2) \mu_1 (d \hat{\gamma}_1) \mu_2 (d
  \hat{\gamma}_2),
\end{equation}
that ought to hold for all $F\in B_{\rm b}(\widehat{\Gamma})$. For
$F^\theta$ as in (\ref{N10j}), it takes the form
\begin{equation}
  \label{N11j}
(\mu_1 \star \mu_2) (F^\theta) = \mu_1 (F^\theta) \mu_2 (F^\theta).
\end{equation}
Recall that a set $\mathcal{F}\subset C_{\rm b}(\widehat{\Gamma})$
is called convergence determining if $\mu_n(F)\to \mu(F)$, $n\to
+\infty$, implies $\mu_n\Rightarrow \mu$, holding for each
$\{\mu_n\}_{n\in \mathds{N}}\subset \mathcal{P}(\widehat{\Gamma})$.
It is known, see \cite[Theorem 4.5, page 113]{EK}, that such
$\mathcal{F}$ has this property if it is closed under pointwise
multiplication and is strongly separating. The latter means that,
for each $\hat{\gamma}\in \widehat{\Gamma}$ and $\epsilon
>0$, there exists a finite family $\{F_j\}\subset \mathcal{F}$ such
that
\begin{equation}
  \label{N11ju}
  \inf_{\hat{\gamma}' \in \widehat{C}_\epsilon} \max_{j} \left|F_j
  (\hat{\gamma}) - F_j(\hat{\gamma}') \right| >0, \qquad
  \widehat{C}_\epsilon := \widehat{\Gamma}\setminus
  \widehat{B}_\epsilon (\hat{\gamma}).
\end{equation}
Here $\widehat{B}_\epsilon (\hat{\gamma})=\{\hat{\gamma}':
\kappa(\hat{\gamma}, \hat{\gamma}') < \epsilon\}$, see (\ref{uM12}).
Note that taking $\epsilon \geq 1$ does not make sense as
$\kappa(\hat{\gamma}, \hat{\gamma}')< 1$ for all configurations.
\begin{proposition}
  \label{N2jpn}
The set $\mathcal{F}_\Theta$ is strongly separating and thus
convergence determining.
\end{proposition}
\begin{proof}
By the very definition of $\Theta$, cf. (\ref{J1}),
$\mathcal{F}_\Theta$ is closed under pointwise multiplication. To
prove (\ref{N11ju}), we note that, see (\ref{N10j}) and
(\ref{uM12z}),
\begin{equation}
  \label{N6j}
  \left|F^{\theta_{s,k,n}}(\hat{\gamma}) - F^{\theta_{s,k,n}}(\hat{\gamma}')
  \right| \geq \min\left\{F^{\theta_{s,k,n}}(\hat{\gamma});
  F^{\theta_{s,k,n}}(\hat{\gamma}')\right\}
  \kappa_{s,k,n}(\hat{\gamma}; \hat{\gamma}')
\end{equation}
holding for all $\hat{\gamma}, \hat{\gamma}'\in \widehat{\Gamma}$.
Now we fix $\hat{\gamma}$ and $\epsilon\in (0,1)$ and then take $m$
such that $2^{-m}< \epsilon /2$. For this $m$ and any $\hat{\gamma'}
\in \widehat{C}_\epsilon$, by (\ref{uM12}) we readily conclude that
\[
\max_{(s,k,n): s+k+n \leq m} \kappa_{s,k,n} (\hat{\gamma},
\hat{\gamma}') >0,
\]
which by (\ref{N6j}) yields the proof.
\end{proof}
The important properties of the family $\mathcal{F}_\Theta$ are
summarized in the following statement.
\begin{proposition}
  \label{N2pn}
The following is true:
\begin{itemize}
  \item[(i)] $\mathcal{B}(\widehat{\Gamma}) = \sigma \{\mathcal{F}_\Theta\}$;
  \item[(ii)] $B_{\rm b}(\widehat{\Gamma})$ is the $bp$-closure of
  the linear span of
  $\mathcal{F}_\Theta$;
\item[(iii)] $\mathcal{F}_\Theta$ is separating; \item[(iv)] $\mathcal{F}_\Theta$ is convergence determining.
\end{itemize}
\end{proposition}
The proof of (i) and (ii) is standard, see \cite[Lemma 3.2.5 and
Theorem 3.2.6, page 43]{Dawson}. The proof of (iv) was done above,
whereas (iii) is a direct consequence of (iv), cf. \cite[page
113]{EK}.
\begin{proposition}
  \label{Ju1pn}
Under Assumption \ref{Jass}, the Kolmogorov operator introduced in
(\ref{S1}) has the property $L:\mathcal{F}_\Theta \to C_{\rm
b}(\widehat{\Gamma})$.
\end{proposition}
\begin{proof}
We consider each of the summands in (\ref{S1}) -- denoted by $L_i$,
$i=1,2,3$ -- separately. Thus,
\begin{gather}
  \label{JJ1}
  (L_3 F^\theta) (\hat{\gamma}) = F^\theta (\hat{\gamma}) \int_X
  \theta (x,0) \chi(d x),
\end{gather}
i.e., it is just the multiplication operator by a $\theta$-dependent
constant. Next,
\begin{equation}
  \label{JJ2}
(L_2 F^\theta) (\hat{\gamma}) = - \sum_{x\in \hat{\gamma}}
m(\hat{x}) \theta (\hat{x}) F^{\theta}(\hat{\gamma}\setminus
\hat{x}) = F^{\theta}(\hat{\gamma})\Psi_2(\hat{\gamma}),
\end{equation}
where
\begin{equation}
  \label{JJ3}
  \Psi_2 (\hat{\gamma}) = \sum_{\hat{x}\in \hat{\gamma}}
\psi_2(\hat{x}) = \sum_{\hat{x}\in \hat{\gamma}} m(\hat{x}) \left(
e^{g(\hat{x})} -1\right).
\end{equation}
Let us consider the following function
\begin{equation*}
  \phi_\tau (\hat{x}) = g(\hat{x}) - \tau m(\hat{x}) \left(
e^{g(\hat{x})} -1\right), \qquad \tau \geq 0.
\end{equation*}
Since each $g_{s,k,n}(\hat{x}) < 1$, see (\ref{uM12z}),  it follows
that $g(\hat{x}) \leq J_\theta$ where $J_\theta$ is just the number
of summands in the sum in (\ref{JJ}). At the same time, $m(\hat{x})
\leq m_*<\infty$, see Assumption \ref{Jass}. Taking this into
account, we set
\begin{equation}
  \label{JJ4a}
\tau_* = 1/ m_* e^{J_\theta}.
\end{equation}
Then
\begin{equation*}
\phi_{\tau_*} (\hat{x}) > 0, \qquad \hat{x}\in \widehat{X}.
\end{equation*}
Now by the simple inequality $\beta e^{-\tau \beta} \leq 1/ e \tau$,
$\tau, \beta >0$, we have, see (\ref{JJ2}) and (\ref{JJ3}),
\begin{gather}
  \label{JJ6}
(L_2 F^\theta) (\hat{\gamma}) \leq  \tau^{-1}_* \exp\left(- 1 -
\sum_{\hat{x}\in \hat{\gamma}} \phi_{\tau_*}(\hat{x})\right) \leq
m_* e^{J_\theta-1},
\end{gather}
which yields the boundedness in question. The continuity of $L_2
F^\theta$ follows by the continuity of $\Psi_2$, which in turn
follows by the fact that $\psi_2\in C_{\rm cs}(\widehat{X})$.
Finally,
\begin{gather*}
  (L_1 F^\theta)(\hat{\gamma}) =\left( - \sum_{\hat{x}\in \hat{\gamma}}
  g'(\hat{x}) \right)F^\theta(\hat{\gamma}) =: \Psi_1
  (\hat{\gamma})F^\theta(\hat{\gamma}).
\end{gather*}
By (\ref{JJ}) we have
\[
g'(x,\alpha) = \sum_{j} v_{s_j}(x) w'_{k_j,n_j} (\alpha),
\]
which yields the continuity of $\Psi_1$. At the same time, by
(\ref{uM2z}) we have
\begin{equation}
  \label{Gt0}|g'(x,\alpha)| \leq \sum_{j} v_{s_j}(x) |w'_{k_j,n_j} (\alpha)|
\leq \bar{\sigma}c g(x,\alpha),
\end{equation}
which yields
\begin{equation}
  \label{Gt}
|g'(x,\alpha)| \leq \bar{\sigma}c J_\theta,
\end{equation}
and also
\[
\left| (L_1 F^\theta)(\hat{\gamma})\right| \leq  \bar{\sigma}c/e.
\]
This completes the proof.
\end{proof}
We summarize the estimates obtained above in the following
\begin{equation}
  \label{Est}
   \sup_{\hat{\gamma}\in \widehat{\Gamma}}|(L F^\theta )(\hat{\gamma})| \leq
   \chi(|\theta(\cdot, 0)|)
    + m_*e^{J_\theta -1} + \bar{\sigma} c/e.
\end{equation}
Note that
\begin{eqnarray}
  \label{Est1}
 \chi(|\theta(x,\alpha)|) & := & \int_{X} |\theta(x,\alpha)| \chi(d x) \\[.2cm] & = & \int_{X} \left(1- e^{-g(x,\alpha)}\right) \chi(d
 x) \leq \int_{X} g(x,0) \chi(d x), \nonumber
\end{eqnarray}
see (\ref{N11jj}).

\section{The Kolmogorov Equation}

\subsection{Notations and useful estimates}
For $\theta \in \Theta$, see (\ref{JJ}), we set
\begin{eqnarray}
  \label{Q8}
  \theta_t (x,\alpha) & = & \theta (x,\alpha + t) \exp\bigg{(}M(x,\alpha) - M(x, \alpha+t)
  \bigg{)}, \\[.2cm] \nonumber
M(x,\alpha) &= & \int_0^{\alpha} m(x, \beta) d \beta.
\end{eqnarray}
Then $\widehat{X}\ni \hat{x} \mapsto \theta_t (\hat{x})$ is
continuous and compactly supported for all $t\geq 0$. Moreover, both
maps $t \mapsto \theta_t (\hat{x})$ and $\alpha \mapsto \theta_t (x,
\alpha)$ are continuously differentiable and the following holds
\begin{equation}
  \label{Di8}
  \frac{\partial}{\partial t} \theta_t (x,\alpha) = \frac{\partial}{\partial \alpha} \theta_t
  (x,\alpha) - m(x,\alpha) \theta_t (x,\alpha).
\end{equation}
Note that
\begin{equation}
  \label{DI}
  (\theta_t)_s (\hat{x}) = \theta_{t+s}(\hat{x}).
\end{equation}
Next, we define, cf. (\ref{JJ}),
\begin{equation}
  \label{JJ8}
  g_t (\hat{x}) = - \log ( 1 + \theta_t(\hat{x})).
\end{equation}
By (\ref{Q8}) it follows that
\begin{equation*}
  |\theta_t(x,\alpha)| \leq |\theta (x, \alpha+t)| = 1 - e^{-g(x,
  \alpha+t)} \leq 1- e^{-J_\theta},
\end{equation*}
where $J_\theta$ is the same as in (\ref{JJ4a}). By (\ref{JJ8}) this
yields
\begin{equation}
  \label{JJ8a}
  g_t (\hat{x}) \leq  J_\theta, \qquad  t\geq 0, \ \ \hat{x} \in
  \widehat{X}.
\end{equation}
By (\ref{Q8}) we also have
\begin{equation*}
  |\theta'_t (x,\alpha) | \leq 2 m_* + |\theta'(x,\alpha+t)| \leq
  2m_* + J_\theta,
\end{equation*}
where $m_*$ is as in (\ref{JJ4a}) and the estimate
$|\theta'(x,\alpha)| \leq |g'(x,\alpha)| \leq J_\theta$ was used,
see (\ref{JJ}) and (\ref{Gt}).

Now we define
\begin{equation}
   \label{Di7}
 F^\theta_t (\hat{\gamma})= \exp\left[\int_0^t \left( \int_X  \theta (x,\alpha) e^{-M(x,\alpha)}  \chi( d x) \right) d
   \alpha \right] F^{\theta_t}(\hat{\gamma}),
 \end{equation}
with $\theta_t$ as in (\ref{Q8}). Clearly, $F^\theta_t\in C_{\rm
b}(\widehat{\Gamma})$ for all $t>0$ and $\theta\in \Theta$, and
\begin{equation}
  \label{Di7y}
0< F^\theta_t (\hat{\gamma}) \leq 1, \qquad \hat{\gamma} \in
\widehat{\Gamma},
\end{equation}
and, for all $t,s\geq 0$, the following holds, see (\ref{DI}),
\begin{equation}
  \label{Di7u}
  F^\theta_{t+s} = \exp\left(\int_0^s \int_X \theta_u(x,0) \chi(d x) d u\right)F^{\theta_s}_t.
\end{equation}
Let us prove that $LF^\theta_t \in C_{\rm b}(\widehat{\Gamma})$ for
all $\theta \in \Theta$ and $t\geq 0$. As in the proof of
Proposition \ref{Ju1pn} we divide $L$ into three parts. Similarly as
in (\ref{JJ1}) we have
\begin{equation*}
  \left|(L_3 F_t^\theta)(\hat{\gamma})  \right| \leq \chi(|\theta
  (\cdot , t)|) \leq \chi(g(\cdot, 0)),
\end{equation*}
see (\ref{Est1}) and (\ref{Di7y}). Since $g_t(\hat{x})$ satisfies
(\ref{JJ8a}) for all $t\geq 0$, it follows that
\[
\left|(L_2 F_t^\theta)(\hat{\gamma})  \right| \leq m_*e^{J_\theta
-1},
\]
holding for all $t>0$, see (\ref{JJ6}). The estimate of $\left|(L_1
F_t^\theta)(\hat{\gamma})  \right|$ is obtained as follows. Denote
\[
q_t(x,\alpha) = \exp\left(-\int_\alpha^{\alpha+t} m(x,\beta) d
\beta\right).
\]
Then by (\ref{Q8}) we have
\begin{equation}
  \label{JJ10}
  \Phi(g_t (x,\alpha)) = q_t(x,\alpha ) \Phi(g(x,\alpha+t)), \qquad
  \Phi (b) := 1-e^{-b}, \ \ b\geq 0,
\end{equation}
by which we get that $g_t (x,\alpha) \leq g (x,\alpha+t)$. Let us
prove that
\begin{equation}
  \label{JJ11}
g_t (x,\alpha) e^{-g_t (x,\alpha)} \geq q_t(x,\alpha )  g
(x,\alpha+t) e^{-g (x,\alpha+t)}.
\end{equation}
By (\ref{JJ10}) this is equivalent to the fact that the function
$b\mapsto b/(e^b -1)$ is decreasing, which is obviously the case.
Now we take the $\alpha$-derivative from both sides of (\ref{JJ10})
and obtain
\begin{gather*}
  g'_t(x,\alpha) = q_t(x,\alpha )  g'
(x,\alpha+t) \exp\bigg{( }g_t(x,\alpha) - g(x, \alpha+t)\bigg{)} \\[.2cm] \nonumber +
m(x,\alpha+t) (e^{g_t(x,\alpha)} -1) - m(x,\alpha)
(e^{g_t(x,\alpha)} -1),
\end{gather*}
that can be estimated as follows
\begin{gather}
  \label{JJ12a}
 | g'_t(x,\alpha)| = \bar{\sigma}c q_t(x,\alpha )
g(x,\alpha+t) \exp\bigg{( }g_t(x,\alpha) - g(x, \alpha+t)\bigg{)} \\[.2cm] \nonumber +
2m_* (e^{g_t(x,\alpha)} -1) \leq g_t(x,\alpha) e^{J_\theta}
(\bar{\sigma}c + 2 m_*),
\end{gather}
where we used (\ref{Gt0}), (\ref{JJ8a}) and (\ref{JJ11}). Now we
proceed as in obtaining (\ref{Est}), which eventually yields
\begin{equation}
  \label{JJ13}
  \left| (LF^\theta_t)(\hat{\gamma})\right| \leq \chi(g(\cdot, 0)) +
  m_*e^{J_\theta-1} + (\bar{\sigma}c + 2m_*)e^{J_\theta}=: \ell_\theta.
\end{equation}
The key property of the latter estimate is that it is uniform in
$t$. However, it does depend on $\theta$.

Along with the estimates derived above, we will use also the
following. For $\theta\in \Theta$, the corresponding $g$ has the
form as in (\ref{JJ}). By (\ref{uM1a}) and (\ref{uM2}) we have that
$$e^{-\bar{\sigma}2^{2/3}/3 }w_{k,n}(0) \leq  w_{k,n}(\alpha)
\leq w_{k,n}(0),$$ which means that, cf. (\ref{N11jj}),
\begin{equation}
  \label{JJ13a}
 \bar{ c} g(x,0) \leq g(x,\alpha) \leq g(x,0), \qquad  \bar{ c} := e^{-\bar{\sigma}2^{2/3}/3
 }.
\end{equation}
For $\Phi$ as in (\ref{JJ10}), we have $b\geq \Phi(b) \geq b -
b^2/2$, $b\geq 0$, which we use together with (\ref{JJ13a}) to
obtain the following
\begin{gather}
  \label{JJ13b}
  g_t(\hat{x}) \geq q_t(\hat{x}) \Phi (\bar{c} g(x,0)) \geq e^{-m_*
  t } \Phi(\bar{c} g(x,0)/J_\theta)\\[.2cm] \nonumber \geq e^{-m_*
  t } \frac{\bar{c} g(x,0)}{J_\theta} \left( 1 - \bar{c}
  g(x,0)/2J_\theta \right)\\[.2cm] \nonumber  \geq e^{-m_*
  t } \bar{c}_\theta g(x,0), \quad \bar{c}_\theta :=
  \bar{c}/2J_\theta,
\end{gather}
where we have used the fact that $g(x,0) \leq J_\theta$, see
(\ref{JJ8a}).
\subsection{The operator}

We fix $\hat{\gamma}\in \widehat{\Gamma}$ and calculate the
$t$-derivative of (\ref{Di8}). This yields
\begin{eqnarray}
  \label{Di9}
\frac{\partial}{\partial t} F^\theta_t (\hat{\gamma}) & = & \left(
\int_X \theta (x,t) e^{-M(x,t)} \chi(d x)\right)F^\theta_t
(\hat{\gamma})  + \sum_{\hat{x}\in \hat{\gamma}} \frac{\partial
\theta (\hat{x})}{\partial t} F^\theta_t (\hat{\gamma}\setminus
\hat{x}) \qquad \\[.2cm] \nonumber & = & \left(
\int_X \theta_t (x,0) \chi(d x)\right)F^\theta_t (\hat{\gamma})  +
\sum_{\hat{x}\in \hat{\gamma}} \frac{\partial }{\partial \alpha_x}
F^\theta_t (\hat{\gamma}) \\[.2cm] \nonumber & + & \sum_{\hat{x}\in
\hat{\gamma}} m(\hat{x}) \left[F^\theta_t (\hat{\gamma}\setminus
\hat{x}) - F^\theta_t (\hat{\gamma})\right] = (L F^\theta_t)
(\hat{\gamma}).
\end{eqnarray}
This means that we have found a solution of the Kolmogorov equation
for (\ref{S1}) in the following sense. It is a map $t\mapsto F_t \in
C_{\rm b}(\widehat{\Gamma})$, which is pointwise in $\hat{\gamma}$
continuously $t$-differentiable and such that the equality in
(\ref{KE}) holds. Our aim now is to solve (\ref{KE}) in a suitable
Banach space.  Recall that the paths $t\mapsto \theta_t$ have the
flow property (\ref{DI}), see also (\ref{Di7u}).

Below, by saying of a property of $\theta_s$, $s\geq 0$, holding for
all $\theta$, we attribute that this property to all $\theta_s$
given in (\ref{Q8}) with $\theta$ taken from $\Theta$.
\begin{proposition}
  \label{Ju2pn}
  For each $\theta\in \Theta$ and $s\geq 0$, it follows that $F_t^{\theta_s} \to
  F^{\theta_s}$ as $t\to 0$ in the norm of $C_{\rm
  b}(\widehat{\Gamma})$.
\end{proposition}
\begin{proof}
For each $\hat{\gamma}$, $s\geq 0$ and $\theta\in \Theta$, by
(\ref{Di9}) it follows that
\begin{gather}
  \label{Yi}
F_t^{\theta_s} (\hat{\gamma}) - F^{\theta_s}(\hat{\gamma}) =
\exp\left(- \int_0^s \int_X \theta_{u}(x,0) \chi(d x) d u
\right)\left[ F^\theta_{t+s} (\hat{\gamma}) - F^\theta_{s}
(\hat{\gamma}) \right] \\[.2cm] \nonumber
= \exp\left(- \int_0^s \int_X \theta_{u}(x,0) \chi(d x) d u \right)
\int_s^{t+s} (LF_u^{\theta}) (\hat{\gamma}) d u,
\end{gather}
which by (\ref{JJ13}) yields
\begin{equation}
  \label{Yie}
\sup_{\hat{\gamma}\in \widehat{\Gamma}} \left|F_t^{\theta_s}
(\hat{\gamma}) - F^{\theta_s}(\hat{\gamma}) \right| \leq t
\ell_\theta \exp\left(- \int_0^s \int_X \theta_u(x,0) \chi(d x) d u
\right).
\end{equation}
This completes the proof.
\end{proof}
The next statement is a refinement of the one just proved.
\begin{proposition}
  \label{Ju3pn}
  For each $\theta\in \Theta$ and $s\geq 0$, it follows that $LF_t^{\theta_s} \to
  LF^{\theta_s}$ as $t\to 0$ in the norm of $C_{\rm
  b}(\widehat{\Gamma})$.
\end{proposition}
\begin{proof}
First of all we note that the equality in the first line of
(\ref{Yi}) allows one to obtain the property in question by showing
that $LF_{t+s}^{\theta} \to
  LF^{\theta}_s$ in the same sense.

As in the proof of Proposition \ref{Ju1pn}, we split $L$ into three
parts and consider each of them separately. Fix $\theta \in \Theta$
and then denote
\begin{equation}
  \label{JJ14}
  \dot{\eta }(u) = \int_X \theta_u(x,0) \chi(dx), \qquad \eta(t)
  =\int_0^t  \dot{\eta}(u) du.
\end{equation}
Then
\begin{eqnarray}
  \label{JJ15}
 \left|(L_3 F_{t+s}^\theta) (\hat{\gamma})- (L_3 F^\theta_s) (\hat{\gamma}) \right|
 & \leq & \left|\dot{\eta}(t+s) e^{\eta(t+s)} - \dot{\eta}(s) e^{\eta(s)} \right| \\[.2cm] \nonumber &
 + &
 |\dot{\eta}(s)| e^{\eta(s)} \left| F^{\theta_{t+s}}(\hat{\gamma})-
 F^{\theta_s}(\hat{\gamma})\right| \\[.2cm] \nonumber & =: & I_1(t) + I_2(t),
\end{eqnarray}
see (\ref{Di7y}). Then $I_1(t)\to 0$ as $t\to 0$ since
$\dot{\eta}(t) e^{\eta(t)}$ is a continuous function of $t$. At the
same time, $I_2(t)$ can be estimated as in (\ref{Yie}). This yields
the proof for $L_3$.

Now we proceed to $L_2$, for which it follows that
\begin{gather}
  \label{JJ16}
(L_2 F_t^{\theta_s}) (\hat{\gamma})- (L_2 F^{\theta_s})
(\hat{\gamma}) = J_t
(\hat{\gamma}) + K_t (\hat{\gamma}), \\[.2cm]\nonumber J_t
(\hat{\gamma}) = \sum_{\hat{x}\in \hat{\gamma}} \frac{m(\hat{x}) (-
\theta_{t+s}(\hat{x}))}{1+ \theta_{t+s}(\hat{x})}[
F^{\theta_s}_t(\hat{\gamma}) - F^{\theta_s}(\hat{\gamma})]
\\[.2cm] \nonumber K_t (\hat{\gamma}) = F^{\theta_s}(\hat{\gamma}) \sum_{\hat{x}\in \hat{\gamma}} \frac{m(\hat{x})
(\theta_s(\hat{x})- \theta_{t+s}(\hat{x}))}{(1+
\theta_{t+s}(\hat{x}))(1+ \theta_s(\hat{x}))}.
\end{gather}
By (\ref{Di7}) and (\ref{JJ14}) we have
\begin{eqnarray}
  \label{JJ17}
F^{\theta_s}_t(\hat{\gamma}) - F^{\theta_s}(\hat{\gamma}) & = &
(e^{\eta(t+s) -\eta(s)}-1) F^{\theta_{t+s}}(\hat{\gamma}) \\[.2cm]
&+ & F^{\theta_{t+s}}(\hat{\gamma}) - F^{\theta_s} (\hat{\gamma})=:
\Upsilon_1 (t, \hat{\gamma}) + \Upsilon_2 (t, \hat{\gamma}).
\nonumber
\end{eqnarray}
By (\ref{Est1}) and (\ref{Q8}), (\ref{JJ14}) it follows that
$$\left|e^{\eta(t+s) - \eta(s)}-1\right| \leq t\chi(g(\cdot, 0)),$$ which then
yields
\begin{equation}
  \label{JJ18}
\left|\Upsilon_1 (t, \hat{\gamma})\right| \leq t\chi(g(\cdot, 0))
\exp\left( - \sum_{\hat{x}\in \hat{\gamma}} g_{t+s}(\hat{x})\right).
\end{equation}
To estimate $\Upsilon_2$ we write
\begin{equation*}
 h_t(\hat{x}) =\min\{g_{t+s}(\hat{x}); g_s(\hat{x})\}.
\end{equation*}
Then by (\ref{JJ12a}) and (\ref{JJ13b}) we get
\begin{gather}
  \label{JJ18b}
\left|\Upsilon_2(t, \hat{\gamma}) \right| \leq \exp\left( -
\sum_{\hat{x}\in \hat{\gamma}} h_t (\hat{x}) \right)\sum_{\hat{x}\in
\hat{\gamma}} \left|g_{t+s} (\hat{x}) -
g_s(\hat{x})\right|\\[.2cm] \nonumber \leq e^{J_\theta} (\bar{\sigma} c + 2 m_*) \exp\left( -
\sum_{\hat{x}\in \hat{\gamma}} h_t (\hat{x}) \right)\sum_{\hat{x}\in
\hat{\gamma}}\int_s^{t+s} g_u(x,\alpha) du, \\[.2cm] \nonumber \leq e^{J_\theta} (\bar{\sigma} c + 2
m_*) \exp\left( -e^{-m_* t} \bar{c}_\theta \sum_{\hat{x}\in
\hat{\gamma}} g(x,0) \right) \sum_{\hat{x}\in \hat{\gamma}} t g(x,0)
\\[.2cm] \nonumber =:t C_\theta e^{-\bar{c}_\theta (t)
\Psi_0(\hat{\gamma})} \Psi_0(\hat{\gamma}), \qquad \bar{c}_\theta
(t) := e^{- m_*t} \bar{c}_\theta, \quad \  \Psi_0 (\hat{\gamma}):=
\sum_{\hat{x}\in \hat{\gamma}} g(x,0).
\end{gather}
At the same time,
\begin{gather*}
 0\leq  \sum_{\hat{x}\in \hat{\gamma}} \frac{m(\hat{x})
 (-\theta_{t+s}(\hat{x})}{1+  \theta_{t+s}(\hat{x})} \leq m_* e^{J_\theta}
 \sum_{\hat{x}\in \hat{\gamma}} g_{t+s}(\hat{x}) =: m_* e^{J_\theta}\Psi_1(t,
 \hat{\gamma}).
 \end{gather*}
Thereafter, we have
\begin{gather*}
 \left|J_t(\hat{\gamma})\right| \leq t m_* e^{J_\theta}\chi(g(\cdot,0))\Psi_1(t,
 \hat{\gamma})  e^{-\Psi_1(t, \hat{\gamma})}
 \\[.2cm] \nonumber
 + t C_\theta   m_* e^{J_\theta} \Psi_1(t, \hat{\gamma}) \Psi_0 (\hat{\gamma}) e^{- \bar{c}_\theta (t) \Psi_0(\hat{\gamma})} =:
 \Pi_1 (t, \hat{\gamma}) + \Pi_2 (t, \hat{\gamma}),
\end{gather*}
where $C_\theta$, $\bar{c}_\theta (t)$ and $\Psi_0$ are as in
(\ref{JJ18b}). Then
\[
\Pi_1 (t, \hat{\gamma}) \leq t m_* e^{J_\theta-1}\chi(g(\cdot,0))
\to 0 , \qquad t\to +\infty.
\]
Let $t_\theta>0$ be the (unique) solution of $\bar{c}_\theta e^{-
m_* t} = 2 t^{1/3}$. Then for $t\leq t_\theta$, we have
\begin{gather}
  \label{JJ21}
\Pi_2 (t, \hat{\gamma}) \leq t C_\theta m_* e^{J_\theta}
[\Psi_0(\hat{\gamma})]^2 e^{-\bar{c}_\theta (t)
\Psi_0(\hat{\gamma})} \\[.2cm] \nonumber \leq t^{1/3} C_\theta m_* e^{J_\theta-2}
\exp\left( - (\bar{c}_\theta(t) - 2 t^{1/3} )
\Psi_0(\hat{\gamma})\right) \leq t^{1/3} C_\theta m_*
e^{J_\theta-2},
\end{gather}
which yields the convergence
\begin{equation*}
  \sup_{\hat{\gamma}\in \widehat{\Gamma}} |J_t(\hat{\gamma})| \to 0,
  \qquad t\to 0.
\end{equation*}
Now we turn to $K_t(\hat{\gamma})$. First, by (\ref{Di8}) we have
\begin{gather}
  \label{JJ21b}
  \left|\theta_s(\hat{x})- \theta_{t+s}(\hat{x}) \right| =
  \left|\int_s^{t+s}
  \left( \frac{\partial}{\partial \alpha} \theta_u (x,\alpha) - m(x,\alpha) \theta_u(x,\alpha)\right) d u
  \right|.
\end{gather}
Next, by (\ref{JJ8}) and (\ref{JJ12a}) it follows that
\begin{gather*}
\left| \frac{\partial}{\partial \alpha} \theta_u (x,\alpha)\right| =
e^{-g_u(x,\alpha)} \left| g'_u (x, \alpha )\right|  \leq
e^{J_\theta}(\bar{\sigma}c + 2 m_*) g_u(x,\alpha) \\[.2cm]\nonumber \leq
e^{J_\theta}(\bar{\sigma}c + 2 m_*) g(x,0),
\end{gather*}
and also
\begin{gather*}
\left|m(x,\alpha) \theta_u (x,\alpha)\right| \leq m_* g_u(x,\alpha)
\leq m_* g(x,0).
\end{gather*}
The latter two estimates yield
\begin{gather*}
  {\rm LHR}(\ref{JJ21b}) \leq e^{J_\theta}(\bar{c} +3m_*) g(x,0).
\end{gather*}
By (\ref{JJ13a}) this yields
\begin{gather*}
  K_t(\hat{\gamma}) \leq t m_* (\bar{\sigma}+ 3 m_*) e^{3J_\theta}
  \Psi_0 (\hat{\gamma}) e^{- \bar{c} \Psi_0(\hat{\gamma})} \\[.2cm]
  \nonumber \leq t m_* (\bar{\sigma}+ 3 m_*) e^{3J_\theta-1}/\bar{c}
  \to 0 , \qquad t\to +\infty.
  \end{gather*}
By (\ref{JJ16}) this completes the proof for $L_2$.

Next, we write
\begin{gather}
  \label{JJ24}
  (L_1 F^{\theta_s}_t)(\hat{\gamma}) - (L_1 F^{\theta_s})(\hat{\gamma})
  = Q_t(\hat{\gamma}) + R_t (\hat{\gamma}), \\[.2cm] \nonumber
  Q_t (\hat{\gamma}) = \sum_{\hat{x}\in \hat{\gamma}}
  \theta'_{t+s}(\hat{x}) \left[F^{\theta_s}_t(\hat{\gamma}\setminus \hat{x}) - F^{\theta_s}(\hat{\gamma}\setminus \hat{x})
  \right], \\[.2cm] \nonumber R_t(\hat{\gamma}) = \sum_{\hat{x}\in
  \hat{\gamma}}\left[\theta'_{t+s}(\hat{x}) - \theta'_s(\hat{x}) \right] F^{\theta_s}(\hat{\gamma}\setminus
  \hat{x}).
\end{gather}
Then
\begin{gather*}
  Q_t(\hat{\gamma}) = - \sum_{\hat{x}\in\hat{\gamma}} g'_{t+s}(\hat{x}) \left[
  F_t^{\theta_s}(\hat{\gamma}) - F^{\theta_s} (\hat{\gamma})\right]\\[.2cm] \nonumber
  + F^{\theta_s}(\hat{\gamma}) \sum_{\hat{x}\in\hat{\gamma}} g'_{t+s}(\hat{x})
  \left[ e^{g_s(\hat{x}) - g_{t+s}(\hat{x}) } - 1\right]=:
  Q^{(1)}_t(\hat{\gamma}) + Q^{(2)}_t (\hat{\gamma}).
\end{gather*}
By (\ref{JJ12a}) and then by (\ref{JJ17}), (\ref{JJ18}),
(\ref{JJ18b}) we get
\begin{gather*}
\left|Q^{(1)}_t(\hat{\gamma})\right| \leq e^{J_\theta} (\bar{\sigma}
c + 2 m_*) \left(\Upsilon_1 (t, \hat{\gamma}) + \Upsilon_2 (t,
\hat{\gamma})\right) \sum_{\hat{x}\in \hat{\gamma}} g_{t+s}
(\hat{x})\\[.2cm] \nonumber \leq t \chi(g(\cdot, 0)) e^{J_\theta} (\bar{\sigma}
c + 2 m_*)\left( \sum_{\hat{x}\in \hat{\gamma}} g_{t+s}
(\hat{x})\right)
\exp\left( - \sum_{\hat{x}\in \hat{\gamma}} g_{t+s} (\hat{x}) \right)\\[.2cm]
\nonumber +t C_\theta e^{J_\theta} (\bar{\sigma} c + 2 m_*)
e^{-\bar{c}_\theta (t) \Psi_0(\hat{\gamma})}
[\Psi_0(\hat{\gamma})]^2 =: \Xi^{(1)}_t (\hat{\gamma}) + \Xi^{(2)}_t
(\hat{\gamma}).
\end{gather*}
Then
\[
\Xi^{(1)}_t (\hat{\gamma}) \leq  t \chi(g(\cdot, 0)) e^{J_\theta-1}
(\bar{\sigma} c + 2 m_*) \to 0, \qquad t\to 0,
\]
and also, cf. (\ref{JJ21}),
\begin{gather*}
\Xi^{(2)}_t (\hat{\gamma}) \leq t^{1/3} C_\theta e^{J_\theta-2}
(\bar{\sigma} c + 2 m_*) \exp\left( - (\bar{c}_\theta(t) - 2 t^{1/3}
) \Psi_0(\hat{\gamma})\right)\\[.2cm] \nonumber \leq t^{1/3} C_\theta m_*
e^{J_\theta-2}(\bar{\sigma} c + 2 m_*),
\end{gather*}
for $t\leq t_\theta$. The latter two estimates yield
\begin{equation}
  \label{JJ28}
  \sup_{\hat{\gamma}\in \widehat{\Gamma}} |Q^{(1)}_t(\hat{\gamma}) |
  \to 0, \qquad t\to 0.
\end{equation}
Next, by (\ref{JJ8a}), (\ref{JJ12a}) and (\ref{JJ13a}) we have
\begin{gather}
  \label{JJ29}
  \left| Q^{(2)}_t(\hat{\gamma})\right| \leq e^{J_\theta}
  F^{\theta_s}(\hat{\gamma}) \sum_{\hat{x}\in \hat{\gamma}}
  |g'_{t+s}(\hat{x})| \left|g_{t+s}(\hat{x}) - g_s(\hat{x}) \right| \\[.2cm]
\nonumber  \leq e^{3J_\theta} (\bar{\sigma} c + 2 m_*)^2
  F^{\theta_s}(\hat{\gamma}) \sum_{\hat{x}\in \hat{\gamma}} g_{t+s}(\hat{x})
  \int_s^{t+s} g_u(\hat{x}) d u \\[.2cm] \nonumber
  \leq t J_\theta e^{3J_\theta} (\bar{\sigma} c + 2 m_*)^2
  \Psi_0(\hat{\gamma}) \exp\left( - \bar{c} \Psi_0(\hat{\gamma})
  \right)\\[.2cm] \nonumber
  \leq t J_\theta e^{3J_\theta} (\bar{\sigma} c + 2 m_*)^2/\bar{c}
  \to 0 , \qquad t\to 0,
\end{gather}
which together with (\ref{JJ28}) yields
\begin{equation}
  \label{JJ30}
  \sup_{\hat{\gamma}\in \widehat{\Gamma}} |Q_t(\hat{\gamma}| \to 0,
  \qquad t\to 0.
\end{equation}
Now we turn to estimating $R_t$. By \ref{Q8}) we have
\begin{gather}
  \label{JJ31}
 \left| \theta_{t+s}'(x,\alpha) - \theta_{s}'(x, \alpha)\right| \leq |\theta'(x,
 \alpha+t+s)|
  \left|q_{t+s}(x, \alpha) - q_{s}(x, \alpha)\right| \\[.2cm] \nonumber + \left| \theta'(x, \alpha+t+s) - \theta'(x, \alpha+s)
  \right| + \left|m(x, \alpha+t+s) - m(x,\alpha+s) \right||\theta_{t+s}(x,
  \alpha)| \\[.2cm] \nonumber + \left|m(x, \alpha+s) - m(x,\alpha) \right||\theta_{s}(x,
  \alpha)|=: \delta_1 (t,\hat{x}) +  \delta_2
  (t,\hat{x}) +  \delta_3 (t,\hat{x}) + \delta_4 (t,\hat{x}).
\end{gather}
By (\ref{JJ12a}) we have
\begin{gather}
  \label{JJ32}
\delta_1 (t,\hat{x}) \leq e^{-g(x,\alpha+t+s)} |g'(x, \alpha +t+s)|
\int_{\alpha+s}^{\alpha+t+s} m(x, \beta ) d \beta \\[.2cm] \nonumber \leq t m_*
e^{J_\theta} (\bar{\sigma} c + 2 m_*) g(x,0) .
\end{gather}
To estimate $\delta_2$, by (\ref{JJ8}) we first get
\[
\theta'(\hat{x}) = - g' (\hat{x}) e^{- g(\hat{x})},
\]
by which we then obtain
\begin{eqnarray*}
\delta_2(t, \hat{x})& \leq & \left| g'(x, \alpha+t+s) -
g'(x,\alpha+s)
\right|e^{-g(x,\alpha+t+s)}  \\[.2cm] \nonumber & + & |g'(x,\alpha+s)|\left|e^{-g(x,\alpha+t+s)} -
e^{-g(x,\alpha+s)} \right| \\[.2cm] \nonumber & \leq & \left| g'(x, \alpha+t+s) - g'(x,\alpha+s)
\right|\\[.2cm] \nonumber & + & |g'(x,\alpha+s)|\left|g(x,\alpha+t+s) - g(x,\alpha+s) \right| \\[.2cm] \nonumber &
=: & \delta_{2,1}(t, \hat{x}) +\delta_{2,2}(t, \hat{x}).
\end{eqnarray*}
Now we recall that $g(x,\alpha)$ is as in (\ref{JJ}) with $w_{k,n}$
defined in (\ref{uM2}). Thus, we can write
\begin{gather}
  \label{JJ32b}
\delta_{2,1}(t, \hat{x}) \leq \sum_{j} v_{s_j}(x) \int_s^{t+s}
|w''_{k_j, n_j} (\alpha + u)| du.
\end{gather}
For each $k$ and $n$, we have
\begin{gather*}
|w''_{k,n}(\alpha)| = \left|- \sigma_k u_n'' (\alpha) e^{-\sigma_k
u_n(\alpha)} + [\sigma_k u_n' (\alpha)]^2 e^{-\sigma_k
u_n(\alpha)}\right| \\[.2cm] \nonumber \leq \bar{\sigma} |u_n''
(\alpha)| + \left[\bar{\sigma} u_n' (\alpha)\right]^2 \leq \bar{C},
\end{gather*}
holding for some $\bar{C}>0$ that is independent of $k,n$ and
$\alpha$. The latter conclusion follows by (\ref{uM2z}) and the fact
that $|u''_n(\alpha)| \leq 2\phi(n\alpha^3)$ with $\phi(\beta) =
(1+\beta^2)/(1+\beta)^3$, $\beta \geq 0$. Then by (\ref{JJ32b}) we
get
\begin{equation}
  \label{JJ32d}
\delta_{2,1}(t, \hat{x}) \leq t \bar{C} g(x,0).
\end{equation}
At the same time, by (\ref{Gt0}) and (\ref{Gt}) it  follows that
\begin{gather*}
\delta_{2,2}(t, \hat{x}) \leq |g'(x,\alpha+s)| \int_s^{t+s}
|g'(x,\alpha+u)| du \leq t(\bar{\sigma}c)^2 J_\theta g(x,0),
\end{gather*}
which together with (\ref{JJ32d}) yields
\begin{equation}
  \label{JJ32f}
  \delta_{2}(t, \hat{x}) \leq t \left[ \bar{C} + (\bar{\sigma}c)^2
  J_\theta \right] g(x,0).
\end{equation}
Finally, (\ref{vkpa}) and (\ref{JJ8}), (\ref{JJ13a}) we have
\begin{gather*}
\delta_3 (t,\hat{x}) \leq \varkappa (t) g_t(\hat{x}) \leq \varkappa
(t) g(x,0).
\end{gather*}
The same estimate holds true also for $\delta_4 (t,\hat{x})$. Then
by (\ref{JJ31}) and (\ref{JJ32}), (\ref{JJ32f}) we have that
\begin{equation*}
 \left| \theta_t'(x,\alpha) - \theta'(x, \alpha)\right| \leq
 \omega(t) g(x,0), \qquad \omega(t)\to 0, \ \   t\to 0.
\end{equation*}
holding for some continuous function $\omega$ and all $\hat{x}\in
\widehat{X}$. Now we use this in (\ref{JJ24}) and obtain, cf.
(\ref{JJ29})
\begin{gather}
  \label{JJ35}
  |R_t(\hat{\gamma})| \leq \omega(t) e^{J_\theta} F^\theta
  (\hat{\gamma}) \Psi_0(\hat{\gamma}) \leq \omega(t) e^{J_\theta}
  \Psi_0(\hat{\gamma}) e^{- \bar{c}\Psi_0(\hat{\gamma})} \leq \omega(t)
  e^{J_\theta}/ \bar{c},
\end{gather}
which together with (\ref{JJ30}) yields
\[
\sup_{\hat{\gamma}\in \widehat{\Gamma}} \left| (L_1 F^\theta_t)
(\hat{\gamma}) - (L_1 F^\theta) (\hat{\gamma}) \right| \to 0, \qquad
t\to 0.
\]
This completes the whole proof.
\end{proof}

\subsection{The domain}

We recall that $\Theta$ consists of the functions as in (\ref{JJ})
and the countable collection $\mathcal{F}_\Theta \subset C_{\rm
b}(\widehat{\Gamma})$ consists of the functions introduced in
(\ref{N10j}). It has a number of useful properties established in
Propositions \ref{N2jpn} and \ref{N2pn}. Let $\mathcal{C}_0$ be the
linear span of the set $\{F^{\theta_s}: s\in \mathds{Q}_+, \
\theta\in \Theta\}$, i.e., each $F\in \mathcal{C}_0$ is a finite
linear combination of $F^{\theta_s}$, with positive rational $s$ and
$\theta_s$ given in (\ref{Q8}) with all possible choices of
$\theta\in \Theta$.
\begin{remark}
  \label{i1rk}
The set $\mathcal{C}_0$ is countable. It enjoys all the properties
mentioned in Proposition \ref{N2pn}.
\end{remark}
Now we set
\begin{equation}
  \label{i0}
  \mathcal{C}= \overline{\mathcal{C}_0},
\end{equation}
i.e., $\mathcal{C}$ is the closure of $\mathcal{C}_0$ in the norm of
$C_{\rm b}(\widehat{\Gamma})$, which we denote $\|\cdot \|$. With
this norm it is then a separable Banach space.

For $\lambda
>0$ and $\theta\in \Theta$ and $s\geq 0$,  we define, cf. (\ref{Di7}),
\begin{eqnarray}
  \label{Do8}
 F_{\lambda,\theta_s}(\hat{\gamma}) & = & \int_0^{+\infty} e^{-\lambda t} F_t^{\theta_s} (\hat{\gamma}) d t \\[.2cm]
  \nonumber & = &\int_{0}^{+\infty}
  \exp\left[ -\lambda t + \int_0^t \left( \int_X \theta_s (x,\alpha) e^{-M(x,\alpha)}\chi(
  d x)
   \right) d \alpha \right] F^{\theta_{t+s}}(\hat{\gamma}) d t.
\end{eqnarray}
Since $\theta_s(x,\alpha)\leq 0$ and  $F_t^{\theta_s}$ satisfies
(\ref{Di7y}), the above integral converges for each $\hat{\gamma}$.
By the dominated convergence theorem and the boundedness
$F^{\theta_s}_t(\hat{\gamma}) \leq 1$ it follows that
$F_{\lambda,\theta_s}\in C_{\rm b}(\widehat{\Gamma})$. Moreover, it
can also be understood as the Bochner integral in the latter Banach
space. Therefore, $F_{\lambda,\theta}$ can be approximated in
$\|\cdot\|$ by the Riemann integral sums centered at rational $t$,
which means that
\begin{equation}
  \label{i}
  F_{\lambda,\theta_s} \in \mathcal{C}, \qquad {\rm for} \ {\rm all}
  \ s\geq 0 \ {\rm and} \
  \lambda >0.
\end{equation}
At the same time, we also have
\begin{equation}
  \label{Do8a}
0 < F_{\lambda,\theta_s}(\hat{\gamma}) < 1/\lambda, \qquad
\hat{\gamma}\in \widehat{\Gamma}.
\end{equation}
By the dominated convergence theorem one readily proves that the map
$(0,+\infty)\ni \alpha \mapsto F_{\lambda,\theta_s}
  (\hat{\gamma}\setminus \hat{x} \cup (x, \alpha))\in \mathds{R}$ is continuously
differentiable for each $\hat{\gamma}$ and $\hat{x}\in
\hat{\gamma}$.
\begin{lemma}
  \label{Di1lm}
For each $\lambda >0$, $s\geq 0$ and $\theta \in \Theta$, the
following holds
\begin{equation}
  \label{Di25}
  L F_{\lambda, \theta_s} = \lambda F_{\lambda, \theta_s} - F^{\theta_s}.
\end{equation}
\end{lemma}
\begin{proof}
By (\ref{Di7}) and (\ref{Do8}) we have
\begin{gather}
\label{Di25a} L F_{\lambda, \theta_s} = L \int_0^{+\infty}
e^{-\lambda t} F^{\theta_s}_t  d t = \int_0^{+\infty}
e^{-\lambda t}  L F^\theta_t d t \\[.2cm] \nonumber =
\int_0^{+\infty} e^{-\lambda t} \frac{\partial}{\partial t}
F^{\theta_s}_t  d t = - F^{\theta_s}  + \lambda F_{\lambda,
\theta_s},
\end{gather}
where we have taken into account the upper bound in (\ref{Do8a}).
The commutation $L \int = \int L$ can be justified by means of the
Lebesgue dominated convergence theorem.
\end{proof}
\begin{lemma}
  \label{i1lm}
For each $\theta\in \Theta$ and $s\geq 0$, it follows that
$\|\lambda F_{\lambda,\theta_s}-F^{\theta_s}\|\to 0$ and $\|\lambda
L F_{\lambda,\theta_s}- L F^{\theta_s}\|\to 0$ as $\lambda \to
+\infty$.
\end{lemma}
\begin{proof}
In view of (\ref{Do8a}), $\{\lambda F_{\lambda, \theta_s}:\lambda
>0\}$ is bounded. By (\ref{Do8}) we have
\begin{gather*}
\lambda F_{\lambda , \theta_s} (\hat{\gamma}) =
\int_0^{+\infty}\exp\left( - t + \int_0^{\varepsilon t}\left( \int_X
\theta_{\alpha+s} (x,0) \chi(d x)\right) d \alpha\right)
F^{\theta_{\varepsilon t+s}}(\hat{\gamma}) dt, \quad \varepsilon
:=\lambda^{-1}.
\end{gather*}
Then by (\ref{Yie}) it follows that
\begin{gather*}
  \left| \lambda F_{\lambda , \theta_s} (\hat{\gamma}) -
  F^{\theta_s}
  (\hat{\gamma})\right| \leq \int_0^{+\infty}e^{-t}\left|  F_{\varepsilon t}^ {\theta_s} (\hat{\gamma}) -
  F^{\theta_s}
  (\hat{\gamma})\right| dt \leq \ell_\theta/\lambda,
\end{gather*}
which yields that $\| \lambda F_{\lambda , \theta_s} -
F^{\theta_s}\|\to 0$ as $\lambda \to+\infty$.  In the same way, by
Proposition \ref{Ju3pn} we have, cf. (\ref{JJ35}),  that
\begin{equation*}
\| \lambda L F_{\lambda , \theta_s}  - L F^{\theta_s}
  \| \leq \tilde{\omega}(1/\lambda) \to 0, \quad
  \lambda \to +\infty,
\end{equation*}
holding for an appropriate continuous $\tilde{\omega}$ such that
$\tilde{\omega}(\epsilon) \to 0$ as $\epsilon \to 0$.
\end{proof}
Let $\mathcal{D}_0(L)$ denote the linear span of the set
$\{F_{\lambda, \theta_s}: \lambda >0, s\geq 0, \theta \in
\varTheta\}$. By Lemma \ref{i1lm} and (\ref{i0}) it follows  that
\begin{equation}
  \label{i6}
  \mathcal{C}_0 \subset \overline{\mathcal{D}_0(L)},
\end{equation}
i.e., $\mathcal{C}_0$ is contained in the closure of
$\mathcal{D}_0(L)$ in the norm of $C_{\rm b}(\widehat{\Gamma})$.
Hence, $\mathcal{D}_0(L)$ is a dense subset of the Banach space
$\mathcal{C}$, see (\ref{i0}). Define
\begin{equation}
  \label{i2}
  \|F\|_L = \|F\| + \|LF\|, \qquad F\in \mathcal{D}_0(L),
\end{equation}
where -- as above -- $\|\cdot\|$ is the norm of $C_{\rm
b}(\widehat{\Gamma})$.
\begin{definition}
  \label{N4df}
By the domain of the Kolmogorov operator $L$, denoted by
$\mathcal{D}(L)$, we mean the closure of $\mathcal{D}_0(L)$ in the
graph-norm introduced in (\ref{i2}).
\end{definition}
\begin{lemma}
  \label{Bo1pn}
The operator $(L, \mathcal{D}(L))$ is closed and densely defined in
$\mathcal{C}$. Its resolvent set contains $(0,+\infty)$ and
$\mathcal{C}_0 \subset \mathcal{D}(L)$.
\end{lemma}
\begin{proof}
In view of (\ref{i}) and (\ref{i6}), the $\|\cdot\|$-closure of
$\mathcal{D}(L)$ is $\mathcal{C}$. The closedness of $(L,
\mathcal{D}(L))$ is immediate, and the inclusion $\mathcal{C}_0
\subset \mathcal{D}(L)$ follows by the second part of Lemma
\ref{i1lm}.
 By (\ref{Di25}) it follows that the
resolvent of $L$, denoted $R_\lambda(L)$, has the property
\begin{equation*}
  R_\lambda (L) F^{\theta_s} = F_{\lambda,\theta_s}, \qquad \lambda >0, \
  \quad s\geq 0, \quad \theta \in \Theta,
\end{equation*}
by which and (\ref{Do8a}) we also have that the operator norm of
$R_\lambda(L)$ satisfies $\|R_\lambda (L) \| \leq 1/\lambda$ as
$F^{\theta_s}$ form a dense subset of $\mathcal{C}$. This completes
the whole proof.
\end{proof}
\subsection{Solving the Kolmogorov equation}

The result obtained in Lemma \ref{Bo1pn} allows one to solve the
Kolmogorov equation (\ref{KE}) in the following sense.
\begin{theorem}
  \label{i1tm}
Let $(L, \mathcal{D}(L))$ and $\mathcal{C}$ be as in Lemma
\ref{Bo1pn}. Then, for each $F\in \mathcal{D}(L)$, there exists a
unique continuously differentiable map $[0,+\infty)\ni t \mapsto F_t
\in \mathcal{D}(L) \subset \mathcal{C}$, which solves (\ref{KE})
with $F_0 = F$. In particular, for $F=F^{\theta_s}$, with
$F^{\theta_s}$ as in (\ref{Q8}), $s\geq 0$ and $\theta\in \Theta$,
the solution has the following explicit form, cf. (\ref{Di7})
\begin{equation}
  \label{i7}
  F_t (\hat{\gamma}) = \exp\left[\int_0^t
  \left(\int_X \theta_s (x,\alpha) \chi(d x) \right) d\alpha \right]
  F^{\theta_{t+s}}(\hat{\gamma}).
\end{equation}
\end{theorem}
\begin{proof}
 By the celebrated Hille-Yosida theorem, see, e.g, \cite[page 8]{Pazy}, $(L, \mathcal{D}(L))$ is the
 generator of a $C_0$-semigroup of bounded linear operators $S(t):
 \mathcal{C}\to \mathcal{C}$ such that $\|S(t)\|= 1$ and the
 solution  in question is $F_t = S(t) F$, the uniqueness of which is
 also a standard fact, see \cite[Theorem 1.3, page 102]{Pazy}. The
 validity of (\ref{i7}) follows by the calculations as in
 (\ref{Di9}).
\end{proof}

\section{The Result}
\subsection{Poisson measures and transition functions}

We begin by recalling, see Proposition \ref{N2pn}, that the class of
functions $\mathcal{F}_\Theta$, see (\ref{N10j}), is separating,
i.e., if $\mu_1(F) = \mu_2(F)$ for all $F\in \mathcal{F}_\Theta$,
then $\mu_1=\mu_2$, that holds for each pair $\mu_1,\mu_2 \in
\mathcal{P}(\widehat{\Gamma})$.  Next, for $t\geq 0$, $\mu\in
\mathcal{P}(\widehat{\Gamma})$ and $\theta\in \Theta$, we define
$\mu^t\in \mathcal{P}({\widehat{\Gamma}})$ by the relation
\begin{equation}
  \label{Q7}
  \mu^t (F^\theta) = \mu(F^{\theta_t}).
\end{equation}
Recall that, for a positive Radon measure $\varrho$ on
$\widehat{X}$, the Poisson measure $\pi_\varrho$ with intensity
measure $\varrho$ is defined in (\ref{N7j}). For $t\geq 0$, we then
introduce the Poisson measure $\pi_t=\pi_{\varrho_t}$ by defining
its intensity measure
\begin{equation}
  \label{Q9}
  \varrho_t (d \hat{x}) = \mathds{1}_{[0,t)} (\alpha) \exp\left(
  - M(\hat{x})\right) \chi( d x) d \alpha,
\end{equation}
where $\chi$ is the same as in (\ref{S1}), $d \alpha$ is the
Lebesgue measure on $\mathds{R}_{+}$ and $M$ is as in (\ref{Q8}).
Note that $\pi_0 (\{\varnothing\}) =1$ since $\pi_0 (F^\theta)=1$
for all $\theta$. Then we set, see (\ref{Conv}),
\begin{equation}
  \label{Q10}
  \mu_t = \pi_t \star \mu^t, \qquad t\geq 0.
\end{equation}
Let $\delta_{\hat{\gamma}}$ be the Dirac measure centered at a given
$\hat{\gamma}\in \widehat{\Gamma}$. Then
\begin{equation}
  \label{Q11}
  p^{\hat{\gamma}}_t  := \pi_t \star \delta^t_{\hat{\gamma}}
\end{equation}
is a transition function, cf. \cite[page 156]{EK}. Indeed,
$p^{\hat{\gamma}}_t\in \mathcal{P}(\widehat{\Gamma})$,
$p^{\hat{\gamma}}_0 = \delta_{\hat{\gamma}}$ and the measurability
of the map $(t,\hat{\gamma}) \mapsto p^{\hat{\gamma}}_t
(B)\in\mathds{R}$, $B \in \mathcal{B}(\widehat{\Gamma})$, follows by
the measurability of $(t,\hat{\gamma}) \mapsto
\delta_{\hat{\gamma}}^t (B)\in\mathds{R}$ and the continuity of $t
\mapsto \pi_t(B)\in\mathds{R}$. In view of the separating property
of $\mathcal{F}_{\varTheta}$, see Proposition \ref{N2pn}, item
(iii), the flow property of $\{p^{\hat{\gamma}}_t\}_{t\geq 0}$ can
be obtained by showing that
\begin{equation}
  \label{Q12}
  p^{\hat{\gamma}}_{t+s} (F^\theta) = \int_{\widehat{\Gamma}_*}  p^{\hat{\gamma}'}_t
  (F^\theta)  p^{\hat{\gamma}}_s ( d \hat{\gamma}'), \qquad t,
  s\geq 0, \
   \ \theta \in \varTheta.
\end{equation}
By (\ref{Conv}) we have
\begin{gather*}
p_t^{\hat{\gamma}'}(F^\theta) = \exp\left( \int_X \int_0^t
\theta_\alpha (x,0) d \alpha \chi(dx)\right)
\delta^t_{\hat{\gamma}'} (F^\theta) \\[.2cm] = \exp\left( \int_X
\int_0^t \theta_\alpha (x,0) d \alpha \chi(dx)\right) F^{\theta_t}
(\hat{\gamma}'),
\end{gather*}
see also (\ref{Q7}) and (\ref{Q8}). By the latter formula and
 (\ref{Q9}), (\ref{Q11}) we then get
\begin{eqnarray*}
{\rm RHS(\ref{Q12})}& = &\exp\left( \int_X \int_0^t \theta_\alpha
(x,0) d \alpha \chi(dx) +  \int_X \int_0^s \theta_{t+\alpha }(x,0) d
\alpha \chi(dx) \right) \delta^s_{\hat{\gamma}} (F^{\theta_t})\qquad \\[.2cm]
& = & \exp\left( \int_X \int_0^{t+s} \theta_\alpha (x,0) d \alpha
\chi(dx) \right) F^{\theta_{t+s}}(\hat{\gamma})  =  {\rm
LHS(\ref{Q12})}.
\end{eqnarray*}
As is known, cf. \cite[Theorem 1.2.page 157]{EK}, the transition
function (\ref{Q11}) determines a Markov process, $\mathcal{X}$,
with values in $\widehat{\Gamma}$, the finite-dimensional
distributions of which are given by the following formula
\begin{gather}
\label{G1} P(\mathcal{X}(s_1)\in B_1, \dots , \mathcal{X}(s_n)\in
B_n)\\[.2cm] \nonumber = \int_{\widehat{\Gamma}} \int_{B_1} \cdots \int_{B_{n-1}}
p_{s_n - s_{n-1}}^{\hat{\gamma}_{n-1}}(B_n)
p_{s_{n-1}-s_{n-2}}^{\hat{\gamma}_{n-2}}(d \hat{\gamma}_{n-1})
\\[.2cm]  \nonumber
\times \cdots \times p_{s_{2}-s_{1}}^{\hat{\gamma}_{1}}(d
\hat{\gamma}_{2}) p_{s_{1}}^{\hat{\gamma}}(d \hat{\gamma}_{1}) \mu
(d\hat{\gamma}),
\end{gather}
holding for all $n\in \mathds{N}$,  $0< s_1 \leq s_2 \leq \cdots
\leq s_n$ and $B_i \in \mathcal{B}(\widehat{\Gamma})$. Here $\mu\in
\mathcal{P}(\widehat{\Gamma})$ is the initial distribution of
$\mathcal{X}$. Our aim is to show that such a process is unique up
to modifications.

\subsection{The statement}

The process determined by (\ref{G1}) describes the stochastic
evolution of the population which we consider. To verify whether it
is the only one, we have to specify which processes of this kind can
be associated to the model defined by the Kolmogorov operator
(\ref{S1}). As is standard, the corresponding specification is made
by their martingale property, see \cite[Chgapter 4]{EK}.
\begin{definition}
  \label{Jundf}
Let $\mathcal{X}$ be a measurable process on some probability space
$(\Omega, \mathfrak{F}, P)$ with values in $\widehat{\Gamma}$.  Let
also $\{\mathfrak{F}_t\}_{t\geq 0}$ be a filtration such that
$\mathcal{X}(t)$ and
\[
\int_0^t G(\mathcal{X}(u)) du
\]
are $\mathfrak{F}_t$-measurable for all $t$ and $G\in
B(\widehat{\Gamma})$. We say that $\mathcal{X}$ is a solution of the
martingale problem for $(L, \mathcal{D}(L))$ if for each $F\in
\mathcal{D}(L)$,
\[
\mathcal{M}(t) := F(\mathcal{X}(t)) - \int_0^t (L F)(\mathcal{X}(u))
d u
\]
is a $\mathfrak{F}_t$-martingale. If there exists a solution of the
martingale problem for $(L, \mathcal{D}(L))$ and uniqueness holds,
we say that the problem is well-posed. In the same way, we define
the martingale problem for $(L, \mathcal{D}(L), \mu)$ if the initial
distribution $\mu\in \mathcal{P}(\widehat{\Gamma})$ is specified.
\end{definition}
The process related to the transition function (\ref{Q11}) solves
the martingale problem for $(L, \mathcal{D}(L))$. Its uniqueness
will be shown by proving that all other solutions have the same
finite-dimensional marginals, i.e., they coincide with those defined
in (\ref{G1}). We are going also to show that the solution is
temporarily ergodic.
\begin{definition}
  \label{T1df}
 Let the martingale problem for $(L,\mathcal{D}(L),\mu)$  be
 well-posed. Then $\mu\in \mathcal{P}(\widehat{\Gamma})$ is said to be a
 stationary distribution if for each $n$ and  $0<s_1 < s_2 < \cdots <s_n$, the $n$-dimensional
 marginals introduced in (\ref{G1}) corresponding to $t+s_1,
 \dots , t+ s_n$ are
 independent of $t\geq 0$.
\end{definition}
If $P$ is as in (\ref{G1}), then $\mu$ is stationary if and only if
\begin{equation*}
  \mu = \int_{\widehat{\Gamma}} p_{s}^{\hat{\gamma}} \mu
(d\hat{\gamma}),
\end{equation*}
holding for all $t>0$. Now we can formulate our result.
\begin{theorem}
  \label{1tm}
The martingale problem for $(L,\mathcal{D}(L))$ is well-posed in the
sense of Definition \ref{Jundf}. Its solution is defined by
finite-dimensional marginals, see (\ref{G1}), with the transition
function defined in (\ref{Q11}). If - in addition to Assumption
\ref{Jass} - the departure function satisfies $m(\hat{x}) \geq m_0
>0$, holding for all $\hat{x}$ and some $m_0$, then there exists a unique stationary distribution
$\mu=\pi_\varrho$, which is the Poisson measure with intensity
measure
\begin{equation}
  \label{Pi}
\varrho(d\hat{x}) = \exp\left( - M(\hat{x})\right) \chi( d x) d
\alpha.
\end{equation}
Moreover, in this case the solution $\mathcal{X}$ of the martingale
problem for $(L, \mathcal{D}(L), \mu)$ is temporarily ergodic in the
following sense. Let  $\mu_t \in \mathcal{P} (\widehat{\Gamma})$ be
the law of $\mathcal{X}(t)$, $t\geq0$. Then $\mu_t \Rightarrow
\pi_\varrho$ as $t\to +\infty$.
\end{theorem}

\section{The Proof}

The proof of Theorem \ref{1tm} is divided into the following steps:
\begin{itemize}
\item[(a)] Proving uniqueness.
\item[(b)] Showing the stationarity and ergodicity if
$m( \hat{x}) \geq m_* >0$.
\end{itemize}
The realization of (a) is based on the Fokker-Planck equation for
$L$, which is a weak version of the forward Kolmogorov equation.

\subsection{The Fokker-Planck equation}

The Fokker-Planck equation is in a sense dual to the Kolmogorov
equation (\ref{KE}), see \cite{FPER} for a general results on such
equations. It describes the evolution of states of our model and
reads
\begin{equation}
  \label{FPE}
  \mu_t (F) = \mu_0 (F) + \int_0^t \mu_s ( LF) d s, \qquad \mu_0 \in \mathcal{P}(\widehat{\Gamma}).
\end{equation}
\begin{definition}
  \label{N4adf}
By a solution of (\ref{FPE}) we understand a map $\mathds{R}_{+}\ni
t \mapsto \mu_t \in \mathcal{P}(\widehat{\Gamma})$ possessing the
following properties: (a) for each $F\in C_{\rm
b}(\widehat{\Gamma})$, the map $\mathds{R}_{+}\ni t \mapsto \mu_t
(F) \in \mathds{R}$ is measurable; (b) the equality in (\ref{FPE})
holds for all  $F\in \mathcal{D}(L)$.
\end{definition}
Note that, for each solution $\mu_t$ and $F\in \mathcal{D}(L)$, the
map $\mathds{R}_{+}\ni t \mapsto \mu_t (F)$ is Lipschitz-continuous.
Hence, it is almost everywhere differentiable, and its derivative is
$\mu_t (LF)$. Therefore
\begin{equation*}
\mu_{t_2} (F) = \mu_{t_1} (F) + \int_{t_1}^{t_2} \mu_s ( LF) d s,
\qquad 0\leq t_1 < t_2.
\end{equation*}
It turns out that, for our model, one can construct solutions of
(\ref{FPE}) explicitly, which we are going to realize now.
\begin{lemma}
  \label{Q3lm}
For each $\mu_0 \in \mathcal{P}(\widehat{\Gamma})$, the map
$t\mapsto \mu_t$ defined in (\ref{Q10}) is a unique solution of
(\ref{FPE}).
\end{lemma}
\begin{proof}
For each $\theta\in \Theta$,  the map $t\mapsto \mu^t(F^\theta)=
\mu_0(F^{\theta_t})$ is continuous (by the dominated convergence
theorem) and hence measurable. By (\ref{N11j}) we then have
\begin{eqnarray*}
 & & \mu_t (F^\theta) =  \pi_t (F^\theta)\mu^t(F^\theta) = \pi_t
 (F^\theta)\mu_0(F^{\theta_t}) \\[.2cm] \nonumber & & \quad  =  \exp\left(
\int_0^t \left[ \int_X \theta (x,\alpha)e^{-M(x,\alpha)} \chi( d x)
\right] d \alpha \right) \mu_0(F^{\theta_t}).
\end{eqnarray*}
Thus, the map $t\mapsto \mu_t(F^\theta)$ is continuous and hence
measurable. Then the measurability of $t\mapsto \mu_t(F)$ for all
$F\in C_{\rm b} (\widehat{\Gamma})$ follows by claim (ii) of
Proposition \ref{N2pn}. Now we turn to proving the equality in
(\ref{FPE}) for $F=F_{\lambda, \theta}$, $\theta\in \varTheta$, see
Definition \ref{N4df}. By (\ref{Q10}) and (\ref{Di7}) for $s, t\geq
0$ and $\theta\in \varTheta$, we have
\begin{gather*}
 \mu_t (F^\theta_s) = \exp\left(\int_0^s \left( \int_X
\theta_\alpha (x,0) \chi(d x)\right) d \alpha \right) \pi_t
(F^{\theta_s}) \mu^t
(F^{\theta_s}) \\[.2cm] \nonumber = \exp\left(\int_0^s \left( \int_X \theta_\alpha
(x,0) \chi(d x)\right) d \alpha + \int_0^t \left( \int_X
\theta_{s+\alpha} (x,0) \chi(d x)\right) d \alpha  \right)
\mu_0(F^{\theta_{s+t}}) \\ \nonumber  = \exp\left(\int_0^{s+t}
\left( \int_X \theta_\alpha (x,0) \chi(d x)\right) d \alpha \right)
\mu_0(F^{\theta_{s+t}}) = \mu_0 (F^\theta_{s+t}).
\end{gather*}
Then by Fubini's theorem and the latter fact we get
\begin{eqnarray*}
& & \mu_t (F_{\lambda, \theta}) - \mu_0 (F_{\lambda, \theta}) =
\int_0^{+\infty} e^{-\lambda s} \left[\mu_t(F^\theta_s) -
\mu_0(F^\theta_s)  \right] d s\\[.2cm] \nonumber & & \ \   = \int_0^{+\infty} e^{-\lambda s}
\mu_0 (F^\theta_{t+s} - F^\theta_s) d s = \int_0^{+\infty}
e^{-\lambda s} \int_0^t \frac{\partial}{\partial u} \mu_0
(F^\theta_{s+u} ) d s d u \\[.2cm] \nonumber & & \ \ = \int_0^{+\infty}
e^{-\lambda s} \int_0^t \frac{\partial}{\partial s} \mu_u
(F^\theta_{s} ) d s d u = \int_0^t \mu_u \left(\int_0^{+\infty}
e^{-\lambda s} \frac{\partial}{\partial s} F^\theta_s d s \right) d
u \\[.2cm] \nonumber & & \ \ = \int_0^t \mu_u  (L
F_{\lambda,\theta}) d u ,
\end{eqnarray*}
where we have used also (\ref{Di25a}). Now we prove uniqueness by
applying arguments similar to those used in \cite[Lemma 2.11]{CK}.
Assume that a map $t\mapsto \mu_t$ satisfies condition (a)
Definition \ref{N4adf} and $F,G \in C_{\rm b}(\widehat{\Gamma})$ are
such that
\begin{equation*}
  \mu_t (F) - \mu_0(F) = \int_0^t \mu_s (G) d s,
\end{equation*}
holding for all $t\geq 0$. Then the map $t\mapsto \mu_t(F)$ is
almost everywhere differentiable and
\[
d \mu_t (F) = \mu_t(G) d t.
\]
Then integrating by parts we get
\begin{gather*}
- \lambda \int_0^t e^{-\lambda s} \mu_s (F) d s = e^{-\lambda t}
\mu_t (F) - \mu_0(F) - \int_0^t e^{-\lambda s} \mu_s (G) d s,
\end{gather*}
which yields
\[
\mu_0 (F) =  e^{-\lambda t} \mu_t (F) + \int_0^t e^{-\lambda s}
\left[ \lambda \mu_s(F) - \mu_s (G)\right] ds,
\]
holding for all $t,\lambda >0$. Passing here to the limit $t\to
+\infty$, for $F= F_{\lambda , \theta}$ and $G = L F_{\lambda ,
\theta}$, see (\ref{FPE}), we arrive at
\begin{equation}
  \label{Di28}
\mu_0 (F_{\lambda , \theta}) = \int_0^{+\infty} e^{-\lambda s} \mu_s
(\lambda F_{\lambda , \theta} - L F_{\lambda , \theta}) d s =
\int_0^{+\infty} e^{-\lambda s}\mu_s (F^\theta) d s,
\end{equation}
see (\ref{Di25}). Assume now that (\ref{FPE}) has two solutions,
$\mu_t$ and $\tilde{\mu}_t$, satisfying the same initial condition
$\mu_t|_{t=0}=\tilde{\mu}_t|_{t=0} = \mu_0$. By (\ref{Di28}) the
Laplace transforms of both maps $t \mapsto \mu_t (F^\theta)$ and $t
\mapsto \tilde{\mu}_t (F^\theta)$ coincide, which yields $\mu_t
(F^\theta)= \tilde{\mu}_t (F^\theta)$ holding for each $t$ and all
$F^\theta$, $\theta \in \Theta$. Then the uniqueness in question
follows by Proposition \ref{N2pn}. This completes the whole proof.
\end{proof}

\subsection{Completing the proof}

The existence of a solution of the martingale problem for
$(L,\mathcal{D}(L))$ was shown by the very construction of the
finite-dimensional marginals of $\mathcal{X}$ in (\ref{G1}). To
prove uniqueness we use the following fact, see \cite[Proposition
4.2, page 184]{EK}. Given $\mu\in \mathcal{P}(\widehat{\Gamma})$,
let $\mathcal{X}$ and $\mathcal{X}'$ be solutions of the martingale
problem for $(L,\mathcal{D}(L),\mu)$ whose one-dimensional
marginals, $\mu_t$ and $\mu'_t$, coincide for all $t\geq0$. Then all
their finite-dimensional marginals coincide and hence the problem is
well-posed. Clearly, both $\mu_t$ and $\mu'_t$ solve the
Fokker-Planck equation with the initial condition $\mu$. Then they
coincide by Lemma \ref{Q3lm}. This yields well-posedness.

Now we show the stated ergodicity. If $m$ satisfies $m(\hat{x})\geq
m_0
>0$, then $M(x,\alpha) \geq m_0 \alpha$, see (\ref{Q8}), which for $\theta$ as in (\ref{Q8}) yields
\begin{equation*}
\forall t >0 \qquad |\theta_t (\hat{x})| \leq  e^{- tm_0}.
\end{equation*}
By the continuity of the map $t \mapsto \mu(F^{\theta_t})$ we then
get
\begin{eqnarray}
  \label{Li}
  \mu_t (F^\theta)& = & \exp\left( \int_0^t \left[ \theta_\alpha (x,0) \chi(d x)\right] d \alpha\right)
\mu(F^{\theta_t}) \\[.2cm] \nonumber
  & \to &\exp\left( \int_{0}^{+\infty} \left[ \int_X \theta_\alpha (x,0) \chi(d x)\right]d \alpha
  \right), \quad t \to +\infty.
\end{eqnarray}
By claim (iv) of Proposition \ref{N2pn} this yields $\mu_t
\Rightarrow \pi_\varrho$, see (\ref{Pi}), holding for each initial
$\mu\in \mathcal{P}(\widehat{\Gamma})$. Clearly, $\mu_t =
\pi_\varrho$ if $\mu = \pi_\varrho$, which means that $\pi_\varrho$
is a stationary state. If there exists another stationary state, say
$\mu'$, then (\ref{Li}) fails to hold for $\mu=\mu'$, which
contradicts the convergence just established. This completes the
proof of Theorem \ref{1tm}.

\section*{Acknowledgements}
The research of both authors was financially supported by National
Science Centre, Poland, grant 2017/25/B/ST1/00051, that is cordially
acknowledged by them.

\section*{Appendix}

Here we prove that $\tilde{r}$ introduced in (\ref{Qh}) satisfies
(\ref{QH}). During the whole proof, we deal with the function
$\ell(\alpha) = \alpha +1/\alpha \geq 2 =\ell(1)$, $\alpha>0$, which
is decreasing for $\alpha\leq 1$ and increasing for $\alpha\geq 1$.

First, consider the case of $\underline{\alpha_2=0}$, where
$\tilde{r}(\alpha_1, \alpha_2)=\omega(\alpha_1)$. For $\alpha_i \leq
1$, $i=1,3$, it follows that $\omega(\alpha_i)=\alpha_i$ and
$\tilde{r}(\alpha_1, \alpha_3) = |\alpha_3-\alpha_1|$. Hence,
(\ref{QH}) turns into
\[
\alpha_1 \leq |\alpha_3-\alpha_1| +\alpha_3,
\]
which obviously holds true. For $\alpha_3>1$,  we have the following
two possibilities: (a) $\alpha_3 - \alpha_1 \leq \alpha_1 +
1/\alpha_3$; (b) $\alpha_3 - \alpha_1 > \alpha_1 + 1/\alpha_3$. In
(a), we get
\begin{gather}
  \label{Qh2}
  \omega(\alpha_1) = \alpha_1 \leq \alpha_3 - \alpha_1 + 1/\alpha_3 =
\tilde{r}(\alpha_1, \alpha_3) + \omega(\alpha_3),
\end{gather}
which yields the inequality in question for this case since
$2\alpha_1 \leq \ell(\alpha_3)$. Note that the equality in
(\ref{Qh2}) is possible only if $\alpha_1=\alpha_3=1$. In (b), we
have $\alpha_1\leq \alpha_1 + 2/\alpha_3$, which completes the proof
of (\ref{QH}) for this case.  For $\alpha_1>1$ and $\alpha_3 \leq
1$, we have the following possibilities: (a) $\alpha_1 -\alpha_3
\leq \alpha_3 + 1/\alpha_1$; (b) $\alpha_1 -\alpha_3 > \alpha_3 +
1/\alpha_1$. In (a), we have
\begin{gather*}
\omega(\alpha_1) = 1/\alpha_1 \leq \alpha_1 = \alpha_1 - \alpha_3 +
\alpha_3,
\end{gather*}
which yields (\ref{QH}) in this case. In case (b), we have
\[
\omega(\alpha_1) = 1/\alpha_1 \leq \alpha_3 + 1/\alpha_1 + \alpha_3.
\]
It remains to consider the cases: (i) $\alpha_1 \geq \alpha_3 > 1$;
(ii) $1 <\alpha_1 < \alpha_3 $. For $\alpha_1 - \alpha_3 \leq
1/\alpha_1 + 1/\alpha_3$, in (i) we have
\[
\omega(\alpha_1) = 1/\alpha_1 \leq 1/\alpha_3 \leq \alpha_1 -
\alpha_3 +1/\alpha_3,
\]
which proves (\ref{QH}) for this case. For $\alpha_1 - \alpha_3 >
1/\alpha_1 + 1/\alpha_3$, in (i) we have
\[
\omega(\alpha_1) = 1/\alpha_1 \leq 1/\alpha_1 + 1/\alpha_3+
1/\alpha_3,
\]
which again yields (\ref{QH}). For $\alpha_3 - \alpha_1 \leq
1/\alpha_1 + 1/\alpha_3$, in (ii) we have to prove $1/\alpha_1 \leq
\alpha_3-\alpha_1 + 1/\alpha_3$, which is equivalent to
$\ell(\alpha_1) \leq \ell(\alpha_3)$. The latter follows by
$\alpha_3 \geq \alpha_1>1$ as $\ell$ is increasing. For $\alpha_3 -
\alpha_1 > 1/\alpha_1 + 1/\alpha_3$, the proof of (\ref{QH}) f is
immediate.

Now we consider the case  $\underline{0<\alpha_1 < \alpha_2}$. If
$\underline{\alpha_2\leq 1}$ and $\alpha_3=0$, then
$\alpha_2-\alpha_1 \leq \omega(\alpha_1) +\omega(\alpha_2)= \alpha_1
+\alpha_2$. For $\alpha_3\in (0,1]$, it follows that
$\tilde{r}(\alpha_3, \alpha_i)=|\alpha_3 - \alpha_i|$. Then
(\ref{QH}) turns into the triangle inequality for $|\cdot|$. The
same is true also for $\alpha_3 >1$ such that $\alpha_3 -1/\alpha_3
\leq 2 \alpha_1$. For $2\alpha_1 <\alpha_3 -1/\alpha_3 \leq 2
\alpha_2$, the right-hand side of (\ref{QH}) is $\alpha_3-\alpha_2 +
\alpha_1 + 1/\alpha_3$. Then $\alpha_2-\alpha_1 \leq {\rm
RHS}(\ref{QH})$ turns into $2(\alpha_2-\alpha_1)\leq
\ell(\alpha_3)$, which holds since $2(\alpha_2-\alpha_1)\leq  2 <
\ell(\alpha_3)$ for $\alpha_3>1$. For $2\alpha_2 <\alpha_3
-1/\alpha_3$, the right-hand side of (\ref{QH}) is $\alpha_1 +
\alpha_2+ 2/\alpha_3$, which is bigger than $\alpha_2-\alpha_1$.

Consider now $\underline{0< \alpha_1  \leq 1 < \alpha_2}$ and
$\underline{\alpha_2 - 1/\alpha_2 \leq 2 \alpha_1}$.  The latter
means that $\tilde{r}(\alpha_2,\alpha_1)= \alpha_2 - \alpha_1$. For
$\alpha_3=0$, the right-hand side of (\ref{QH}) is $\alpha_1
+1/\alpha_2$, and the latter turns into $\alpha_2 - 1/\alpha_2 \leq
2 \alpha_1$, which holds in this case. The same is true for
$2\alpha_3 \leq \alpha_2 - 1/\alpha_2$. For $\alpha_2-1/\alpha_2
\leq 2\alpha_3\leq 2\alpha_1$, the right-hand side of (\ref{QH}) is
$\alpha_2-\alpha_3 + \alpha_1-\alpha_3$, which is bigger than
$\alpha_2-\alpha_1$ as $\alpha_3\leq \alpha_1$. For $\alpha_1 \leq
\alpha_3 \leq 1$, the right-hand side of (\ref{QH}) is
$\alpha_2-\alpha_3 + \alpha_3-\alpha_1 = {\rm LHS}(\ref{QH})$, Next,
consider $1<\alpha_3\leq \alpha_2$, where $\alpha_3-1/\alpha_3 \leq
\alpha_2-1/\alpha_2 \leq 2\alpha_1$. For $\alpha_2-1/\alpha_2 \leq
\ell(\alpha_3)$, the right-hand side of (\ref{QH}) is
$\alpha_2-\alpha_3+ \alpha_3-\alpha_1 = {\rm RHS}(\ref{QH})$. The
case of $\alpha_2-1/\alpha_2 > \ell(\alpha_3)>2$ is impossible since
$\alpha_2 -1/\alpha_2\leq 2 \alpha_1 \leq 2$. For $\alpha_3 >
\alpha_2$ such that $\alpha_3 -1/\alpha_3 \leq 2 \alpha_1$ and
$\alpha_3 -1/\alpha_3 \leq \ell(\alpha_2)$, we have
$\tilde{r}(\alpha_3,\alpha_1)= \alpha_3 - \alpha_1$ and
$\tilde{r}(\alpha_3,\alpha_2)= \alpha_3 - \alpha_2$. Then the
right-hand side of (\ref{QH}) is $\alpha_3-\alpha_1 +\alpha_3
-\alpha_2 = 2\alpha_3 - (\alpha_2+\alpha_1)$, which is bigger than
$\alpha_2 - \alpha_1$. The case of $\alpha_3 -1/\alpha_3 >
\ell(\alpha_2)$ is impossible for $\alpha_3-1/\alpha_3 \leq 2
\alpha_1$. For $\alpha_3 > \alpha_2$ such that $\alpha_3 -1/\alpha_3
> 2 \alpha_1$ and $\alpha_3 -1/\alpha_3 \leq \ell(\alpha_2)$, we
have $\tilde{r}(\alpha_3,\alpha_1) = \alpha_1 + 1/\alpha_3$ and
$\tilde{r}(\alpha_3,\alpha_2) = \alpha_3 -\alpha_2$. Then the
right-hand side of (\ref{QH}) is $\ell(\alpha_3)-\alpha_2 +
\alpha_1$, which yields (\ref{QH}) in the form $2(\alpha_2-\alpha_1)
\leq \ell(\alpha_3)$. By $\alpha_2 - 1/\alpha_2 \leq 2\alpha_1$,
$\alpha_2>1$, we have that $2(\alpha_2-\alpha_1) \leq 2
\sqrt{1+\alpha_1^2}$, whereas $\alpha_3 -1/\alpha_3>2\alpha_1$
yields $\ell(\alpha_3)> \ell(\alpha_*)= 2 \sqrt{1+\alpha_1^2}$,
which proves (\ref{QH}) in this case. Here $\alpha_*$ is the
positive solution of $\alpha-1/\alpha = 2 \alpha_1$. For $\alpha_3 >
\alpha_2$ such that $\alpha_3 -1/\alpha_3
> 2 \alpha_1$ and $\alpha_3 -1/\alpha_3 > \ell(\alpha_2)$, we
have $\tilde{r}(\alpha_3,\alpha_1) = \alpha_1 + 1/\alpha_3$ and
$\tilde{r}(\alpha_3,\alpha_2) = 1/ \alpha_3 + 1/\alpha_2$. Then
(\ref{QH}) turns to
\[
\alpha_2 - \alpha_1 \leq \alpha_1 + 2/\alpha_3 + 1/\alpha_2,
\]
which holds as $\alpha_2-1/\alpha_2 \leq 2 \alpha_1$.

Consider now $\underline{0< \alpha_1  \leq 1 < \alpha_2}$ and
$\underline{\alpha_2 - 1/\alpha_2 > 2 \alpha_1}$. The latter means
that $\tilde{r}(\alpha_1,\alpha_2)=\alpha_1 + 1/\alpha_2$. For
$\alpha_3 \in [0,\alpha_1]$, we have that
$\tilde{r}(\alpha_1,\alpha_3)=\alpha_1-\alpha_3$ and
$\tilde{r}(\alpha_2,\alpha_3)=\alpha_3+1/\alpha_2$. Hence,
(\ref{QH}) turns into equality. For $\alpha_3 \in (\alpha_1, 1]$
such that $\alpha_2-1/\alpha_2 > 2\alpha_3$, we have the right-hand
side of (\ref{QH}) in  the following form $\alpha_3 -\alpha_1 +
\alpha_3 + 1/\alpha_2$, which is bigger that $\tilde{r}(\alpha_1,
\alpha_2)$ since $\alpha_1 < \alpha_3$. For $\alpha_3 \in (\alpha_1,
1]$ such that $\alpha_2-1/\alpha_2 \leq 2\alpha_3$, we have that the
right-hand side of (\ref{QH}) is $\alpha_3 -\alpha_1 + \alpha_2
-\alpha_3 = \alpha_2 - \alpha_1 > \tilde{r}(\alpha_1, \alpha_2)$.
Consider now $\alpha_3> 1$ such that $\alpha_3 -1/\alpha_3 \leq 2
\alpha_1$, which means that $\alpha_3< \alpha_2$ and
$\tilde{r}(\alpha_1, \alpha_3) = \alpha_3 - \alpha_1$. For
$\tilde{r}(\alpha_2, \alpha_3) = \alpha_2 -\alpha_3$, the right-hand
side of (\ref{QH}) is $\alpha_2 -\alpha_3 +\alpha_3-\alpha_1\geq
\tilde{r}(\alpha_2, \alpha_3)$. For $\tilde{r}(\alpha_2, \alpha_3) =
1/\alpha_3 +1/\alpha_2$, the right-hand side of (\ref{QH}) is
$\alpha_3 -\alpha_1 +1/\alpha_2 + 1/\alpha_3$; hence, (\ref{QH})
turns into $2\alpha_1 \leq \ell(\alpha_3)$, which holds since
$\alpha_1\leq 1$ and $\ell(\alpha_3)\geq 2$. For $\alpha_3 >1$ such
that $\alpha_3 -1/\alpha_3 > 2 \alpha_1$, we have that
$\tilde{r}(\alpha_1, \alpha_3) = \alpha_1 + 1/\alpha_3$. Then
(\ref{QH}) turns into
\begin{equation}
  \label{QH1}
1/\alpha_2 \leq 1/\alpha_3 + \tilde{r}(\alpha_2 , \alpha_3),
\end{equation}
which clearly holds for $\alpha_3\leq \alpha_2$, and also for
$\alpha_3
>\alpha_2$, where for $\tilde{r}(\alpha_2 , \alpha_3) =
\alpha_3-\alpha_2$ it turns into $\ell(\alpha_2) \leq
\ell(\alpha_3)$ -- which is true as $\ell(\alpha)$ is increasing for
$\alpha>1$. For $\tilde{r}(\alpha_2 , \alpha_3) = 1/\alpha_3
+1/\alpha_2$, the validity (\ref{QH1}) is immediate.

Let us consider now the case of $\underline{1<\alpha_1  < \alpha_2}$
and $\underline{\ell(\alpha_1)\leq \alpha_2 - 1/\alpha_2}$, where
$\tilde{r}(\alpha_2 , \alpha_1) = 1/\alpha_1 + 1/\alpha_2$. For
$\alpha_3 \leq 1$ such that $2\alpha_3 \leq \alpha_1 - 1/\alpha_1$,
we have $\tilde{r}(\alpha_3 , \alpha_i) = \alpha_3 + 1/\alpha_i$,
$i=1,2$. Then (\ref{QH}) obviously holds. For $\alpha_3 \leq 1$
satisfying $\alpha_1 - 1/\alpha_1 <2\alpha_3 \leq \alpha_2 -
1/\alpha_2$, we have $\tilde{r}(\alpha_3 , \alpha_1) = \alpha_1 -
\alpha_3$ and $\tilde{r}(\alpha_3 , \alpha_2) = \alpha_3
+1/\alpha_2$. Hence, (\ref{QH}) turns into equality in this case.
The remaining case $\alpha_2 - 1/\alpha_2 <2\alpha_3 \leq 2$ is
impossible since $\alpha_2 - 1/\alpha_2 \geq \ell(\alpha_1) > 2$.
For $\alpha_3>1$ such that $2\alpha_3 \leq \alpha_1 - 1/\alpha_1$,
we have $\tilde{r}(\alpha_3 , \alpha_i) = 1/\alpha_3 + 1/\alpha_i$,
$i=1,2$. Then (\ref{QH}) obviously holds. For $\alpha_3\in (1,
\alpha_1]$ satisfying $\alpha_1 - 1/\alpha_1< 2\alpha_3 \leq
\alpha_2 - 1/\alpha_2$, (\ref{QH}) turns into
\[
1/\alpha_1 + 1/\alpha_2 \leq \alpha_1 - \alpha_3 +1/\alpha_2 +
1/\alpha_3,
\]
which holds for $\alpha_3\leq \alpha_1$ as the function $\alpha
-1/\alpha$ is increasing. For $\alpha_3 \in (\alpha_1, \alpha_2]$
such that $\alpha_3-1/\alpha_3 \leq \ell(\alpha_1) < \ell(\alpha_3)
< \alpha_2-1/\alpha_2$, the right-hand side of (\ref{QH}) is
$\alpha_3 - \alpha_1 + 1/\alpha_3 +1/\alpha_2$, and hence the latter
turns into $\ell(\alpha_1) < \ell(\alpha_3)$. For $\alpha_3 \in
(\alpha_1, \alpha_2]$ satisfying $\alpha_3-1/\alpha_3 \leq
\ell(\alpha_1)$ and $\alpha_2-1/\alpha_2 \leq \ell(\alpha_3)$, we
have (\ref{QH}) in the form $1/\alpha_2 +1/\alpha_1 \leq \alpha_2
-\alpha_3 +\alpha_3 -\alpha_1$, which holds as $\ell(\alpha_1) <
\alpha_2-1/\alpha_2$. For $\alpha_3 \in (\alpha_1, \alpha_2]$
satisfying $\ell(\alpha_1)<\alpha_3-1/\alpha_3$ and
$\alpha_2-1/\alpha_2 \leq \ell(\alpha_3)$, the right-hand side of
(\ref{QH}) is $1/\alpha_1 + 1/\alpha_3 + \alpha_2 - \alpha_3$ which
is bigger than $\tilde{r}(\alpha_1 , \alpha_2)$ as $\alpha -
1/\alpha$ is increasing. For $\alpha_3>\alpha_2$ such that $\alpha_3
-1/\alpha_3 \leq \ell(\alpha_2)$, the right-hand side of (\ref{QH})
is $1/\alpha_1 + 1/\alpha_3 + \alpha_3 -\alpha_2$. Hence, (\ref{QH})
holds as $\ell(\alpha_2) < \ell(\alpha_3)$. For $\alpha_3
-1/\alpha_3 > \ell(\alpha_2)$,  (\ref{QH}) turns into $1/\alpha_1 +
1/\alpha_2\leq  1/\alpha_1 + 1/\alpha_3 + 1/\alpha_3 +1/\alpha_2$.

Finally, let us consider the case of $\underline{1<\alpha_1  <
\alpha_2}$ and $\underline{\ell(\alpha_1)> \alpha_2 - 1/\alpha_2}$,
where $\tilde{r}(\alpha_2 , \alpha_1) = \alpha_2-\alpha_1$. For
$\alpha_3\leq 1$ such that $2\alpha_3 \leq \alpha_1-1/\alpha_1$, it
follows that $\tilde{r}(\alpha_i,\alpha_3)= \alpha_3 +1/\alpha_i$,
$i=1,2$. Then (\ref{QH}) turns into $\alpha_2 - 1/\alpha_2 \leq
\ell(\alpha_1) + 2 \alpha_3$, which evidently holds in this case.
For $\alpha_1 - 1/\alpha_1 < 2 \alpha_3 \leq \alpha_2 - 1/\alpha_2$,
we have that $\tilde{r}(\alpha_1 , \alpha_3) = \alpha_1-\alpha_3$
and $\tilde{r}(\alpha_2 , \alpha_3) = \alpha_3 + 1/\alpha_2$. Then
(\ref{QH}) turns into $\alpha_2 - 1/\alpha_2 \leq 2 \alpha_1$, which
is the case for $2\alpha_1 > \ell(\alpha_1) > \alpha_2 -
1/\alpha_2$. For $\alpha_3\leq 1$ satisfying $2\alpha_3 > \alpha_2 -
1/\alpha_2$, it follows that $\tilde{r}(\alpha_i,\alpha_3)= \alpha_i
-\alpha_3$, $i=1,2$. Then (\ref{QH}) turns into $2 \alpha_3 \leq 2
\alpha_1$. which is obviously the case. For $\alpha_3 \in
(1,\alpha_1]$ such that $\ell(\alpha_3) \leq \alpha_1-1/\alpha_1$,
we have that $\tilde{r}(\alpha_i,\alpha_3)= 1/\alpha_i +1/\alpha_3$,
$i=1,2$. Then (\ref{QH}) amounts to $\alpha_2 - 1/\alpha_2 \leq
\ell(\alpha_1) + 2 /\alpha_3$, which obviously holds. For
$\alpha_1-1/\alpha_1<\ell(\alpha_3)\leq  \alpha_2-1/\alpha_2$, we
have that $\tilde{r}(\alpha_1,\alpha_3)= \alpha_1 -\alpha_3$ and
$\tilde{r}(\alpha_2,\alpha_3)= 1/\alpha_2 +1/\alpha_3$. Then
(\ref{QH}) amounts to $\alpha_1 - \alpha_3 + \alpha_1 +
1/\alpha_3\geq \alpha_2 - 1/\alpha_2$, which is the case for
$\alpha_1 - \alpha_3 + \alpha_1 + 1/\alpha_3\geq \alpha_1 +
1/\alpha_3 \geq \ell(\alpha_1) > \alpha_2 -1/\alpha_2$. For
$\alpha_2-1/\alpha_2<\ell(\alpha_3)$, we have
$\tilde{r}(\alpha_i,\alpha_3)= \alpha_i - \alpha_3$, $i=1,2$. Then
(\ref{QH}) is $\alpha_2 - \alpha_1 \leq \alpha_2-\alpha_3 + \alpha_1
- \alpha_3$, which obviously holds as $\alpha_3 \leq \alpha_1$. Now
we consider $\alpha_3 \in (\alpha_1, \alpha_2]$. For $\alpha_3
-1/\alpha_3 \leq \ell(\alpha_1)$, we have
$\tilde{r}(\alpha_1,\alpha_3)= \alpha_3 -\alpha_1$ and
$\tilde{r}(\alpha_2,\alpha_3)= \alpha_2 -\alpha_3$, which yields
equality in (\ref{QH}). Recall that $\alpha_2 - 1/\alpha_2 <
\ell(\alpha_1)$; hence, $\alpha_3 - 1/\alpha_3 > \ell(\alpha_1)$ is
impossible for $\alpha_3 \leq \alpha_2$. It remains to consider
$\alpha_3 > \alpha_2$. For $\alpha_3 -1/\alpha_3\leq
\ell(\alpha_1)$, we have that $\tilde{r}(\alpha_i,\alpha_3)=
\alpha_3 -\alpha_i$, $i=1,2$. Then (\ref{QH}) takes the form
$\alpha_2 - \alpha_1 \leq 2 \alpha_3 - \alpha_1 - \alpha_2$, which
obviously holds in this case. For $\ell(\alpha_1)< \alpha_3
-1/\alpha_3\leq \ell(\alpha_2)$, (\ref{QH}) amounts to $2\alpha_2
\leq \ell(\alpha_3) +  \ell(\alpha_1)$, which holds as
$\ell(\alpha_1)> \alpha_2 - 1/\alpha_2$ (assumed) and
$\ell(\alpha_3)> \alpha_2 + 1/\alpha_2$ for $\alpha_3>\alpha_2$. For
$\alpha_3 -1/\alpha_3> \ell(\alpha_2)$, we have that
$\tilde{r}(\alpha_3, \alpha_i) = 1/\alpha_3 + 1/\alpha_i$, $i=1,2$.
Then (\ref{QH}) turns into
\[
\alpha_2 - \alpha_1 \leq 2/\alpha_3 + 1/\alpha_2 +1/\alpha_1,
\]
which holds since $\ell(\alpha_1) > \alpha_2 -1/\alpha_2$. This
completes the whole proof.


\begin{thebibliography}{9}






\bibitem{FPER} V. I. Bogachev, N. V. Krylov, M. R\"ockner, S. V. Shaposhnikov,
Fokker-Planck-Kolmogorov Equations. Mathematical Surveys and
Monographs, 207. American Mathematical Society, Providence, RI,
2015.







\bibitem{CK} C. Constantini, T. G. Kurtz, Viscosity methods giving
uniqueness for martingale problems. Electron. J. Probab. 67 (2015),
1--27.

\bibitem{DV2} D. J. Daley, D. Vere-Jones,  An Introduction to the Theory of Point Processes. Vol. II. General Theory and Structure.
Second edition. Probability and its Applications (New York).
Springer-Verlag, New York, 2008.


\bibitem{Dawson} D. A. Dawson, Measure-Valued Markov Processes. {\'E}cole d'{\'E}t{\'e} de
Probabilit{\'e}s de Saint-Flour XXI--1991, 1--260, Lecture Notes in
Math., 1541, {\it Springer, Berlin,} 1993.




\bibitem{EK} S. N. Ethier, T. G. Kurtz, Markov Processes: Characterization and Convergence, {\it Wiley, New York,} 1986.





\bibitem{He} Z-R. He, D. Ni, Sh. Wang, Existence and stability of steady states for hierarchical age-structured population models.
Electron. J. Differential Equations No. 124 (2019), 14 pp.

\bibitem{Fima} P. Jagers, F. Klebaner,  Population-size-dependent and age-dependent branching processes, Stochastic Process. Appl. 87 (2000),  235--254.





\bibitem{Koz1} Y. Kozitsky, Stochastic branching at the edge: Individual-based modeling of tumor cell
proliferation, J. Evol. Equ. {\bf 21} (2021) 2081--2104.



\bibitem{Lenard} A. Lenard, Correletion functions and the uniqueness
of the state in classical statistical mechanics, {\it Comm. Math.
Phys.} {\bf 30} (1973) 35--44.




\bibitem{bib7} S. M\'el\'eard, V. Tran, Slow and fast scales for superprocess limits of age-structured populations,  Stochastic Process. Appl. 122 (2012), 250--276.

\bibitem{Part} K. Parthasarathy, Probability Measures on Metric
Spaces, Probability and Mathematical Statistics, No. 3, {\it
Academic Press, Inc., New York -- London,} 1967.

\bibitem{Pazy} A. Pazy, Semigroups of Linear Operators and Applications to Partial Differential
Equations, Applied Mathematical Sciences, Vol. 44 Springer-Verlag,
New York, 1983.



\bibitem{Wang} Z. Wang, Stability and moment boundedness of an age-structured model with randomly-varying immigration or harvesting, J. Math. Anal. Appl. 471 (2019), 423--447.
\bibitem{Zessin} Zessin, H. The method of moments for random measures. Z. Wahrscheinlichkeitstheorie verw Gebiete
62, 395--409 (1983).

\end{thebibliography}
\end{document}